\documentclass[reqno]{amsart}

\usepackage{amssymb,enumerate,mathrsfs}

\numberwithin{equation}{section}
\newtheorem{theorem}[equation]{Theorem}
\newtheorem{proposition}[equation]{Proposition}
\newtheorem{lemma}[equation]{Lemma}
\theoremstyle{definition}
\newtheorem{definition}[equation]{Definition}

\DeclareMathOperator{\Diff}{Diff}
\DeclareMathOperator{\End}{End}
\DeclareMathOperator{\spec}{spec}
\DeclareMathOperator{\supp}{supp}
\DeclareMathOperator{\sym}{ \sigma\!\!\!\sigma}

\def\C{\mathbb C}
\def\R{\mathbb R}

\def\m{\mathfrak m}
\def\K{\mathcal K}
\def\scrH{\mathscr H}
\def\L{\mathscr L}

\def\M{\mathcal M}
\def\N{\mathcal N}
\def\Y{\mathcal Y}
\def\Z{\mathcal Z}

\def\agen{\mathfrak a}
\def\ext{\mathfrak e}
\def\Ext{\mathfrak E}
\def\gen{\mathfrak g}
\def\trb{\mathscr T}
\def\preP{\mathfrak P}
\def\Dom{\mathcal D}
\def\Ring{\mathcal R}

\def\minus{\backslash}
\def\im{i}
\def\Id{I}

\def\open#1{\smash[t]{\overset{{}_{\,\,\circ}}{#1}{}}}
\def\set#1{\left\{#1\right\}}

\def\rpar{)}
\def\lbra{[}

\def\ev{\mathrm{ev}}
\def\wev{\,{}^w\hspace{-0.5pt}{\ev}}
\def\wT{\,{}^w\hspace{-0.5pt}T}
\def\wpi{\,{}^w\!\pi}
\def\wsym{\,{}^w\!\!\sym}
\def\wg{\,{}^w\!g}
\def\bP{\,{}^b\!P}
\def\bPhat{\,{}^b\!\widehat P}
\def\bA{\,{}^b\!A}

\def\loc{\mathrm{loc}}
\def\sc{\mathrm{sc}}

\def\ft{\!\widehat{\phantom{m}}}

\begin{document}

\title[The Friedrichs extension]{The Friedrichs extension for elliptic wedge operators of second order}
\thanks{Work partially supported by the National Science Foundation, Grants DMS-0901202 and DMS-0901173. The authors are grateful to \emph{Humboldt-Universit{\"a}t zu Berlin} and the \emph{SFB 647: Space--Time--Matter. Analytic and Geometric Structures} for the hospitality and support during September 2013.}

\author{Thomas Krainer}
\address{Penn State Altoona\\ 3000 Ivyside Park \\ Altoona, PA 16601-3760}
\email{krainer@psu.edu}
\author{Gerardo A. Mendoza}
\address{Department of Mathematics\\ Temple University\\ Philadelphia, PA 19122}
\email{gmendoza@temple.edu}

\begin{abstract}
Let $\M$ be a smooth compact manifold whose boundary is the total space of a fibration $\N\to \Y$ with compact fibers, let $E\to\M$ be a vector bundle. Let 
\begin{equation*}\tag{\dag}
A:C_c^\infty(\open \M;E)\subset x^{-\nu} L^2_b(\M;E)\to x^{-\nu} L^2_b(\M;E)
\end{equation*}
be a second order elliptic semibounded wedge operator. Under certain mild natural conditions on the indicial and normal families of $A$, the trace bundle of $A$ relative to $\nu$ splits as a direct sum $\trb=\trb_F\oplus\trb_{aF}$ and there is a natural map $\preP:C^\infty(\Y;\trb_F)\to C^\infty(\open\M;E)$ such that $C^\infty_{\trb_F}(\M;E)=\preP(C^\infty(\Y;\trb_F))+\dot C^\infty(\M;E)\subset \Dom_{\max}(A)$. It is shown that the closure of $A$ when given the domain  $C^\infty_{\trb_F}(\M;E)$ is the Friedrichs extension of (\dag) and that this extension is a Fredholm operator with compact resolvent. Also given are theorems pertaining the structure of the domain of the extension which completely characterize the regularity of its elements at the boundary.
\end{abstract}


\maketitle


\section{Introduction}

The purpose of this paper is to give an explicit description of the domain of the Friedrichs extension of a second order semibounded elliptic wedge operator initially defined on smooth functions or sections with compact support away from the boundary.

Generally, the issue of characterizing the domain of the Friedrichs extension only arises in the presence of singularities, either in the form of singular coefficients of a partial differential operator, for instance Schr{\"o}dinger operators with singular (magnetic) potentials, or, for geometric operators, due to the incompleteness of the underlying Riemannian geometry. Indeed, in a seminal paper Gaffney \cite{Gaffney} proved the essential selfadjointness of the Laplacian on complete Riemannian manifolds, a result extended by Shubin \cite{Shub92} and Kordyukov \cite{Ko91} to wide classes of uniformly elliptic operators in $L^2$ on complete manifolds with bounded geometry. Shubin and Kordyukov showed that uniformly elliptic operators are essentially closed, i.e., $\Dom_{\min} = \Dom_{\max}$, where that domain is independent of the operator and in fact is a Sobolev space determined by the geometric data of the manifold.

Since Cheeger's papers \cite{Ch79,Ch83} it has been understood that a relevant category of singular partial differential equation problems consists of those that give rise to iterated wedge (or iterated incomplete edge \cite{AlbinLeichtMazzPiazza12}) operators on stratified spaces. The simplest instances are wedge operators on manifolds with smooth singular strata (singularities of edge type); these specialize further to cone operators if the singular strata consist only of points (conical singularities). The analysis of operators on a general stratified space is in principle amenable to inductive arguments once the principles in the case of conical and edge singularities are understood.

In the case of manifolds with conical singularities, the domain of the Friedrichs extension of semibounded elliptic cone operators was characterized in full generality in \cite{GiMe01}; previous investigations include \cite{BrSee1,Ch79,Le97}. These works go further in that they also provide a structural understanding of the maximal domain of an elliptic cone operator, initially given as
$$
A : C_c^{\infty} \subset x^{-\nu}L^2_b \to x^{-\nu}L^2_b,
$$
in terms of an explicit description of the minimal domain and a complementary space. We paraphrase this in the form that there exists an exact sequence
\begin{equation*}
0\to \Dom_{\min} \xrightarrow{\iota} \Dom_{\max} \xrightarrow{\gamma} {\mathscr S} \to 0,
\end{equation*}
where ${\mathscr S}$ is a \emph{specific} finite-dimensional space of singular functions associated to the indicial roots of $A$ encoding the asymptotic behavior of the elements in the maximal domain at the boundary modulo remainders in the minimal domain. $\Dom_{\min}$ is described in terms of weighted $b$-Sobolev spaces (\cite{Mel93}), for instance $\Dom_{\min} = x^{-\nu+2}H^2_b$ if $A$ is of second order and the line $\{\sigma: \Im(\sigma) = \nu-2\}$ is free of indicial roots. The map $\gamma$ provides a concrete realization of the abstract Cauchy data map $\Dom_{\max} \to \Dom_{\max}/\Dom_{\min}$ in terms of the space ${\mathscr S}$. For semibounded operators $A$, it is made precise in \cite{GiMe01} that ${\mathscr S} = {\mathscr S}_F \oplus {\mathscr S}_{aF}$ canonically, where ${\mathscr S}_{aF}$ consists of the `most singular half' of ${\mathscr S}$, and the domain of the Friedrichs extension, $\Dom_F$, is given by
$$
\Dom_F = \{u \in \Dom_{\max} : \gamma(u) \in {\mathscr S}_F\}.
$$
Equivalently, we find that $u$ is in $\Dom_F$ if and only if $u$ satisfies the boundary condition that the component of $\gamma(u)$ in ${\mathscr S}_{aF}$ vanishes, and we get an exact sequence
\begin{equation}\label{ResF}
0\to \Dom_{\min} \xrightarrow{\iota} \Dom_{F} \xrightarrow{\gamma} {\mathscr S}_F \to 0.
\end{equation}
This explicit formulation provides a complete structural understanding of $\Dom_F$. 

The boundary condition that is entailed by membership in the Friedrichs domain can be phrased in several different ways without explicit reference to the space ${\mathscr S}$, for example as a condition on the blow-up rate of $u$ at the boundary (as is done in \cite[Theorem~6.1]{BrSee1}). However, without \eqref{ResF}, the boundary condition and asymptotic behavior that is imposed on $u$ by membership in $\Dom_F$ remain elusive.

The case of edge singularities considered in the present paper is significantly more complicated for reasons that pertain to regularity. The abstract Cauchy data space $\Dom_{\max}/\Dom_{\min}$ of an elliptic wedge operator $A$ is in general infinite dimensional, and comprehending its structure is made complicated first and foremost by the fact that the indicial roots of $A$, parametrized by the edge $\Y$, generally vary with arbitrary branchings. This imposes the need to measure variable and anisotropic Sobolev smoothness of functions with respect to the edge variables in the analysis of regularity at the boundary.

One of the central results of this paper is a full structural resolution of the Friedrichs domain of a semibounded elliptic wedge operator $A$ of second order, under some mild assumptions on the symbols, generalizing \eqref{ResF} to the situation at hand. The space ${\mathscr S}_F$ is shown to be a Sobolev space of sections of varying regularity of a certain subbundle $\trb_F$ of the trace bundle $\trb \to \Y$ that is associated to $A$ (\cite{KrMe12a,KrMe12b}), while $\Dom_{\min} = x^{-\nu+2}H^2_e$ (cf. \cite{GiKrMe10,Mazz91}).

Previous work on describing the structure of the Friedrichs domain for wedge operators includes \cite{BruSee91,MazzVert12,MelVasWunEdge} and has centered on the boundary condition that is imposed on a function by membership in $\Dom_F$ for certain special operators of second order. In \cite[Theorem~5.1]{BruSee91} the condition is phrased in terms of a control of the blow-up rate at the boundary along with conditions pertaining to the derivatives. \cite[Section~5]{MelVasWunEdge} considers the Friedrichs \emph{form domain} for the scalar Laplacian and describes it in terms of edge Sobolev spaces. \cite[Proposition 2.5]{MazzVert12} treats the full Hodge Laplacian under the additional assumption that the indicial roots are constant, and identifies the Friedrichs boundary condition in terms of vanishing of certain terms in a weak asymptotic expansion at the boundary which has previously been shown to exist in \cite{Mazz91}. The argument in both \cite{MelVasWunEdge} and \cite{MazzVert12} is aided by the fact that the operator is of the form $A = B^{\star}B$ for a first order operator $B$. A structural resolution \eqref{ResF} of the Friedrichs domain is not addressed in these papers.

Other related work pertaining to the analysis of regularity at the boundary for solutions to elliptic equations on manifolds with edges and more general stratified manifolds includes \cite{AmLaNi07,CostabelDauge,KrMe13,MazRossmann,Mazz91,MazzVert14,Mel93,SchuNH,SchulzeIterated,SchVol}.


\section{Overview}\label{sec-Overview}

We work on a smooth compact manifold $\M$ with boundary $\N$, where the latter is the total space of a fibration $\wp:\N\to\Y$ over some other manifold $\Y$ with compact fibers. This configuration results from blowing up a stratified manifold with a single singular stratum $\Y$, along the singularity. The manifold $\Y$ need not be connected, but we will assume it is for the sake of notational simplicity. Associated with this configuration and a vector bundle $E\to\M$ is the ring of edge differential operators, $\Diff^*_e(\M;E)$; these are linear differential operators with smooth coefficients up to the boundary, which along the latter differentiate only in directions tangent to the fibers. For all of the above and more see Mazzeo \cite{Mazz91}. 

The elements of $x^{-m}\Diff^m_e(\M;E)$ are wedge operators; $x$ is a defining function for $\N$, positive in $\open\M$, which we fix from now on. Associated to such a wedge operator $A=x^{-m}P$ there are three symbols: the wedge-principal symbol $\wsym(A)$, the indicial family $\bPhat$, and the normal family $A_{\wedge}$. The first is essentially the standard principal symbol of $A$ over the interior, the other two are objects  over the boundary. For an in-depth description of the differential topological setup in relation with wedge operators, including invariant definitions of the various symbols in the $w$ context, see \cite{GiKrMe10}. Some details will be given in the next section.

Recall (\cite{Kato}) that the Friedrichs extension of a symmetric, densely defined operator $A : \Dom \subset H \to H$ acting in a Hilbert space $H$ that is semibounded from below, where without loss of generality $A \geq 1$, is the selfadjoint operator $A_F : \Dom_F \subset H \to H$ given by $A_F = A^{\star}$ with domain $\Dom_F = \Dom(A^{\star})\cap\scrH$, where $\scrH \subset H$ is the completion of $\Dom$ with respect to the inner product $[u,v] = \langle Au,v \rangle$.

Let $\m$ be a smooth positive density on $\M$ and $\m_b=x^{-1}\m$. Suppose $E$ is given a Hermitian metric, so the spaces $x^{-\nu}L^2_b(\M;E)$ of $L^2$ sections of $E$ with respect to $x^{2\nu}\m_b$ are defined. 

Let $A\in x^{-2}\Diff^2_e(\M;E)$ be given. The basic assumptions we make on $A$ are  ellipticity (invertibility of $\wsym(A)$), symmetry, and semiboundedness, the latter two being  properties of $A$ as an operator
\begin{equation}\label{InitialOperatorDag}\tag{\dag}
A:C_c^\infty(\open \M;E)\subset x^{-\nu} L^2_b(\M;E)\to x^{-\nu} L^2_b(\M;E).
\end{equation}
Additionally we place conditions on the boundary symbols, briefly described in the next paragraph (see Section~\ref{sec-SetupAssumptions} for details). The task is to give a description of the domain of the Friedrichs extension of \eqref{InitialOperatorDag} as explicitly as possible. For symmetry reasons it is more convenient to work with $x^{-\nu+1}Ax^{\nu-1}$, so we replace $A$ by the latter and consider
\begin{equation}\label{InitialOperator}
A:C_c^\infty(\open \M;E)\subset x^{-1} L^2_b(\M;E)\to x^{-1} L^2_b(\M;E)
\end{equation}
instead of \eqref{InitialOperatorDag}; this does not affect the hypotheses originally made on $A$.

The indicial family of $A$ is the indicial family of $P=x^2A$. This is a family $\bPhat(y,\sigma)\in \Diff^2(\Z_y;E_{\Z_y})$ of elliptic operators parametrized by $(y,\sigma)\in \Y\times\C$, depending polynomially on $\sigma$, of order $2$; here $\Z_y=\wp^{-1}(y)$. Because of the $w$-ellipticity of $A$, for each $y\in \Y$ the differential operator $\bPhat(y,\sigma)$ fails to be invertible for $\sigma$ in a set $\spec_{b,y}(A)$ whose intersection with every strip $|\Im\sigma|<r$ is finite. The condition we place on $\bPhat$ is that
\begin{equation*}
\spec_e(A)=\set{(y,\sigma)\in \Y\times\C:\sigma \in \spec_{b,y}(A)}
\end{equation*}
shall not contain points $(y,\sigma)$ with $\sigma$ on one of the three lines $\Im\sigma=1$, $0$, $-1$ (the pertinent lines for \eqref{InitialOperatorDag} are  $\Im\sigma=\nu$, $\nu-1$, $\nu-2$). Let now $\pi_\N:\N^\wedge\to\N$ be the inward pointing normal bundle of $\N\subset\M$. Then $\wp_\wedge=\wp\circ\pi_\N$ is also a fibration. Let $\Z_y^\wedge=\wp_\wedge^{-1}(\Z_y)$. So $\Z_y^\wedge\approx \Z_y\times\lbra0,\infty\rpar$, and quantizing we get the operators $\bA(y,D_x)=x^{-2}\bPhat(y,xD_x)$ acting on sections of $E_{\Z_y^\wedge}$. The consequence pertaining the outermost lines $\Im\sigma=\pm1$ is that the vector spaces $\trb_y$ whose elements are the sections of $E_{\Z_y^\wedge}\to\Z_y^\wedge$ in $\ker \bA(y,D_x)$ of the form
\begin{equation}\label{TraceElement}
\hspace*{-8pt}\sum_{\sigma \in \spec_{b,y}(A)\cap\Sigma} \sum_{k=0}^{m_{\sigma}} c_{\sigma,k}\log^k(x)x^{i\sigma},\quad\Sigma=\set{\sigma:-1<\Im\sigma<1}
\end{equation}
have constant dimension as $y$ ranges throughout $\Y$, and are the fibers of a smooth vector bundle $\trb\to\Y$, the trace bundle of $A$. The coefficients $c_{\sigma,k}$ are smooth sections of $E_{\Z_y}$. For details on vector bundles associated in this fashion with holomorphic Fredholm families we direct the reader to \cite{KrMe12a}. The condition pertaining the line $\Im\sigma=0$ determines a splitting $\trb=\trb_F\oplus\trb_{aF}$ in which the elements of $\trb_{F,y}$ are those of the form \eqref{TraceElement} over $y$ with $\Sigma$ replaced by $\Sigma_F=\set{\sigma:-1<\Im\sigma<0}$, while those in $\trb_{aF}$ are defined using $\Sigma_{aF}=\set{\sigma:0<\Im\sigma<1}$ instead.

The normal family is a family parametrized by the elements of $T^*\Y$, with $A_{\wedge}(\pmb\eta)\in x^{-2}\Diff_b^2(\Z_y^\wedge;E_{\Z_y^\wedge})$ for every $\pmb\eta\in T^*_y\Y$. On this family we place the requirement that $0$ is not an eigenvalue of the Friedrichs extension of
\begin{equation*}
A_\wedge(\pmb \eta):C_c^\infty(\open \Z_y^\wedge;E_{\Z_y^\wedge})\subset x^{-1} L^2_b(\Z_y^\wedge;E_{\Z_y^\wedge})\to x^{-1} L^2_b(\Z_y^\wedge;E_{\Z_y^\wedge})
\end{equation*}
when $\pmb\eta\ne 0$. This condition implies in particular that the minimal domain of $A$ (the domain of the closure of \eqref{InitialOperator}) is $x^1H^2_e(\M;E)$ by \cite[Theorem 4.2]{GiKrMe10}, and is used in an essential manner in Section~\ref{sec-ExtensionOperator} to ensure the existence of certain inverses. The weighted edge Sobolev spaces $x^kH^s_e(\M;E)$ were defined in \cite{Mazz91}. Consistent with our previous works \cite{GiKrMe10,KrMe13} we follow here a slightly different convention pertaining to the use of weights in that we base the definition on $L^2_b$ instead, i.e., we have $H^0_e(\M;E) = L^2_b(\M;E)$.

Let $\omega\in C^\infty(\R)$ be smooth, equal to $1$ near $0$, with small support. A smooth section $\phi$ of $C^\infty(\Y;\trb_F)$ is by definition (of smoothness, see \cite{KrMe12a}) a smooth section of $\open\pi_{\N}^*E\to\open \N^\wedge$ ($\open\pi_\N$ being the restriction of $\pi_\N$ to $\open\N$). The section $\omega\phi$ can be transferred to a section of $E$ on $\M$ supported near $\N$. Write $\preP(\phi)$ for the resulting section. Theorem~\ref{BasicFriedrichsTheorem} asserts that the domain of the Friedrichs extension of \eqref{InitialOperator}, denoted $\Dom_F$, is the closure in the $A$-graph norm (based on $x^{-1}L^2_b$) of the space
\begin{equation}\label{FriedrichsCore}
C^\infty_{\trb_F}(\M;E)=\preP\big(C^\infty(\Y;\trb_F)\big)+\dot C^\infty(\M;E).
\end{equation}
Theorem~\ref{FriedrichsIsFredholm} asserts that $A : \Dom_F(A) \to x^{-1}L^2(\M;E)$ is Fredholm with compact resolvent.

The sum in \eqref{FriedrichsCore} is in fact direct, therefore gives rise to an exact sequence
\begin{equation}\label{SmoothShortExact}
0\to \dot C^\infty(\M;E) \xrightarrow{\iota}C^\infty_{\trb_F}(\M;E)\xrightarrow{\gamma}C^\infty(\Y;\trb_F)\to 0.
\end{equation}
The operator $\gamma$ is a generalization of part of the classical Cauchy data (or trace) map for the standard theory of boundary value problems for regular elliptic operators on manifolds with smooth boundary. Classically, for second order operators the Cauchy data are the pair consisting of the restrictions to the boundary of the function (or section) $u$, $\gamma_0(u)$, and its normal derivative, $\gamma_1(u)$. These are the coefficients of the linear part of the Taylor series at the boundary with respect to the defining function $x$ of a putative solution on which boundary conditions are to be placed. Assembled into the first order polynomial $\gamma_0(u)x^0+\gamma_1(u) x^1$, they are generalized by expressions of the form \eqref{TraceElement}, here with $\sigma\in \set{0, -\im}$. Classically, the Friedrichs domain corresponds to the condition $\gamma_0(u)=0$, that is, the only possible nonzero coefficient, $\gamma_1(u)$, is the one associated to the lower half of $\set{0, -\im}$.

\smallskip
The operators in \eqref{SmoothShortExact} extend to continuous operators on the closure of the various spaces appearing in the sequence, each within its appropriate Hilbert space, and yield the short exact sequence,
\begin{equation}\label{SobolevShortExact}
0\to x^1H^2_e(\M;E) \xrightarrow{\iota}\Dom_F(A)\xrightarrow{\gamma}H^{2-\gen}(\Y;\trb_F)\to 0
\end{equation}
alluded to in the introduction (see Theorem~\ref{ExactSequenceFriedrichsTheorem}). The space $H^{2-\gen}(\Y;\trb_F)$ is a Sobolev space of varying anisotropic regularity to be described in some detail in a moment. The space $x^1\!H^2_e(\M;E)$ is $\Dom_{\min}(A)$, as pointed out above. 

The exactness of the sequence \eqref{SobolevShortExact} is a regularity statement for the problem 
\begin{equation*}
Au=f\in x^{-1}L^2_b(\M;E),\quad u\in \Dom_F
\end{equation*}
since a solution $u$ {\it a fortiori} satisfies $\gamma(u)\in H^{2-\gen}(\Y;\trb_F)$; this characterizes the behavior of $u$ modulo $\Dom_{\min}(A)$.
In the classical $L^2$-theory of regular elliptic operators of second order on compact manifolds with smooth boundary, the sequence \eqref{SobolevShortExact} specializes to
\begin{equation*}
0\to H_0^2(\M;E) \xrightarrow{\iota} \Dom_F(A) \xrightarrow{\gamma_1}H^{1/2}(\partial\M;E)\to 0,
\end{equation*}
which resembles the classical regularity result that $\Dom_F(A) \subset H^2(\M;E)$, i.e.,
\begin{equation*}
\Dom_F(A) = \set{u \in H^2(\M;E) : u|_{\partial\M} = 0};
\end{equation*}
note that $\trb_F \cong E|_{\partial\M}$ canonically via $\gamma_1(u)x^1 \leftrightarrow \gamma_1(u)$, and the base Hilbert space is $L^2 = x^{-1/2}L^2_b$ as opposed to $x^{-1}L^2_b$ in this case.

We now describe $H^{2-\gen}(\Y;\trb_F)$. The vector bundle $\trb\to\Y$ comes equipped with a natural endomorphism, namely $x\partial_x$: since $x\partial_x$ commutes with $x^2\bA$, it preserves the kernel of the latter and evidently it preserves the structure of the elements of the form \eqref{TraceElement}. For vector bundles $G\to\Y$ with endomorphism $\agen$ we defined in \cite{KrMe12b} Sobolev spaces of variable (anisotropic) order $\agen$, denoted by $H^\agen(\Y;G)$. The idea is that the local regularity of a generalized eigensection of $\agen$ is roughly measured by the real part of the eigenvalue. In the present situation the relevant space is $H^{2-\gen}(\Y;\trb_F)$ with
\begin{equation}\label{generator}
\gen = 1 + x\partial_x \in \End(\trb_F).
\end{equation}
The reason for the shift by $1$ in $\gen$ comes from working with $x^{-1}L^2_b$ as base Hilbert space: it has the effect of making the action $\varrho^\gen$ (see \eqref{kappaaction}) unitary. For general $\nu$ as in \eqref{InitialOperatorDag} we would use $\nu+x\partial_x$. 

\medskip

The analytic machinery needed to prove our theorems in Section~\ref{sec-MainResult} is developed in Section~\ref{sec-ExtensionOperator}. The proof utilizes the basic functional analytic principle that if the densely defined operator $A : \Dom \subset H \to H$ is symmetric, and if $A + \mu : \Dom \to H$ is invertible for some value of the real parameter $\mu$, then $A$ with domain $\Dom$ is selfadjoint.

The principal objective in Section~\ref{sec-ExtensionOperator} is the construction and analysis of a family of extension operators
\begin{equation*}
\Ext_\lambda:C^{-\infty}(\Y;\trb_F)\to C^\infty(\open \M;E), \; \lambda \in \R,
\end{equation*}
closely related to $\preP$ with a number of desirable properties; for instance, $(\Ext_\lambda-\preP)|_{C^\infty(\Y;\trb_F)}$ maps into $\dot C^\infty(\M;E)$,
\begin{align}
&\text{there exists }\delta_0 > 0\text{ such that }\Ext_{\lambda} : H^{2-\gen}(\Y;\trb_F) \to x^{\delta_0}H^{\infty}_e(\M;E), \label{introeq1} \\
&\Ext_{\lambda} - \Ext_{\lambda'} : H^{2-\gen}(\Y;\trb_F) \to x^1H^{\infty}_e(\M;E) \text{ for } \lambda,\lambda' \in \R
\label{introeq2}
\end{align}
by Lemma~\ref{ExtBasicMapProps}, and Lemma~\ref{ExtAdvMapProps} states that
\begin{equation}\label{introeq3}
(A + \lambda^2) \circ \Ext_{\lambda} : H^{2-\gen}(\Y;\trb_F) \to x^{-1}H^{\infty}_e(\M;E)\text{ for all }\lambda\in \R.
\end{equation}
By \eqref{introeq2} the domain
$$
\Dom = \Ext_{\lambda}\big(H^{2-\gen}(\Y;\trb_F)\big)+x^1H^2_e(\M;E)
$$
is independent of $\lambda \in \R$, \eqref{introeq1} implies that it is contained in the Friedrichs form domain $\scrH$ by means of a duality argument (see the proof of Theorem~\ref{StructureFriedrichsDomain}), and \eqref{introeq3} shows that it is contained in the maximal domain
\begin{equation*}
\Dom_{\max}(A) = \set{u \in x^{-1}L^2_b(\M;E) :  Au \in x^{-1}L^2_b(\M;E)}.
\end{equation*}
As pointed out earlier, $\Dom_F = \scrH \cap \Dom_{\max}(A)$, and consequently $\Dom \subset \Dom_F$. We then proceed to show that $A + \lambda^2 : \Dom \to x^{-1}L^2_b(\M;E)$ is invertible whenever $|\lambda|$ is large enough, see Proposition~\ref{Parametrix}, which implies that $\Dom = \Dom_F$. The proof of the invertibility utilizes a parameter-dependent parametrix and the specific structure of the extension operators $\Ext_{\lambda}$, in particular their dependence on the parameter $\lambda \in \R$.


\section{Standing assumptions}\label{sec-SetupAssumptions}

Henceforth we fix a second order wedge operator $A \in x^{-2}\Diff^2_e(\M;E)$ on a compact manifold with fibered boundary as described previously. The operator shall be viewed as an unbounded operator
\begin{equation}\label{AinfixedL2}
A : C_c^{\infty}(\open\M;E) \subset x^{-1}L^2_b(\M;E) \to x^{-1}L^2_b(\M;E).
\end{equation}
In applications one may instead view $A$ as acting on $x^{-\nu}L^2_b(\M;E)$ for some $\nu\in\R$. For instance, in geometric situations involving the Laplacian with respect to a $w$-metric $\wg$, a natural choice is $\nu=(\dim\Z+1)/2$ in view of
\begin{equation*}
L^2_{\wg}(\M;E) = x^{-(\dim\Z+1)/2}L^2_b(\M;E).
\end{equation*}
However, we will assume $\nu = 1$ in the sequel by replacing $A$ with $x^{-\nu+1}Ax^{\nu-1}$ as explained in the previous section.   The definition of inner product on the spaces $x^{-\nu}L^2_b(\M;E)$ is such that conjugation by powers of the boundary defining function is a unitary equivalence, so conditions such as symmetry or semiboundedness are preserved.

The standing assumptions on $A$ in Sections \ref{sec-MainResult} and \ref{sec-ExtensionOperator} will be
\begin{enumerate}[\quad (i)]
\item\label{SymmetricSemibounded} $A$ in \eqref{AinfixedL2} is symmetric and semibounded from below by $1$;
\item\label{w-ellipticity} $A$ is $w$-elliptic, i.e., the $w$-principal symbol $\wsym(A)$ is invertible on $\wT^*\M\minus 0$;
\item\label{b-strips} $\spec_e(A) \cap [\Y\times\{\sigma \in \C :  \Im(\sigma) \in \{-1,0,1\}\}] = \emptyset$;
\item\label{AhatInvertible} for each $\pmb \eta\in T^*\Y\minus 0$, zero is not an eigenvalue of the Friedrichs extension of $A_\wedge(\pmb\eta)$.
\end{enumerate}

\medskip
We elaborate on these conditions. The first condition of course means that 
\begin{equation*}
\langle Au,v \rangle_{x^{-1}L^2_b} = \langle u,Av \rangle_{x^{-1}L^2_b}
\end{equation*}
and
\begin{equation*}
\langle Au,u \rangle_{x^{-1}L^2_b} \geq  \langle u,u \rangle_{x^{-1}L^2_b}
\end{equation*}
for all $u,v \in C_c^{\infty}(\open\M;E)$. Note that if initially $A\geq c$ for some $c\in \R$ then replacing $A$ by $A + (1-c)\textup{Id}$ allows us to assume $c=1$ --- neither the symbolic assumptions (\ref{w-ellipticity})-(\ref{AhatInvertible}) nor the conclusions are affected by this change.

In \cite[Section 2]{GiKrMe10} we defined the wedge-cotangent bundle, $\wpi:\wT^*\M\to\M$ as the bundle whose smooth sections are the sections of $T^*\M$ which along the boundary are conormal to each fiber of $\wp$. There is a bundle map $\wev^*:\wT^*\M\to T^*\M$ over the identity which is an isomorphism over $\open \M$. The relevancy of this bundle lies in the fact that the principal symbol of $A$ is naturally a section of $\End(\wpi^*E)$. We elaborate. The differentials of functions in the ring
\begin{equation*}
\Ring = \{f \in C^{\infty}(\M) : f\big|_{\N}=\wp^*g \textup{ for some } g \in C^{\infty}(\Y)\}
\end{equation*}
generate the space of sections of $\wT^*\M$ as a module over $C^{\infty}(\M)$, see \cite{KrMe13}. Suppose $f \in \Ring$ and $\phi \in C^{\infty}(\M;\wT^*\M)$ corresponds to $df$ via $\wev^*\phi = df$. Let $\psi \in C_c^{\infty}(\open\M;E)$. Then, over $\open\M$,
\begin{equation*}
\wsym(A)(\phi)\psi = \lim_{\varrho \to \infty}\varrho^{-2}e^{-i\varrho f}A(e^{i\varrho f}\psi).
\end{equation*}
Consequently, the natural notion of ellipticity for $A$ is $w$-ellipticity, i.e., invertibility of $\wsym(A)$.

Let  $P = x^2A$, so $P \in \Diff^2_e(\M;E)$. The latter is a subset of $ \Diff^2_b(\M;E)$, so $P$ restricts to an operator $\bP:C^\infty(\N;E_\N)\to C^\infty(\N;E_\N)$ which differentiates only in the direction of the fibers of $P$. (Here and elsewhere the notation $E_W$ means the part of the vector bundle $E$ over the set or the point $W$.) The same is true for  $P_\sigma=x^{-\im\sigma}Px^{\im\sigma}\in \Diff^2_b(\M;E)$, and the indicial families of $P$ and of $A$ are defined to be both $\bP_\sigma$. This is a family $\bPhat(y,\sigma) : C^{\infty}(\Z_y;E_{\Z_y}) \to C^{\infty}(\Z_y;E_{\Z_y})$ depending on $(y,\sigma) \in \Y\times\C$, and assumption (\ref{b-strips}) means that we require it to be invertible for all $(y,\sigma) \in \Y \times\set{\sigma: \Im(\sigma) \in \set{-1,0,1}}$.

Let $\pi_\N:\N^\wedge\to\N$ be the inward pointing normal bundle of $\N$ (including the zero section which we identify with $\N$) and let $\wp_\wedge:\N^\wedge\to\Y$ be the obvious map. $\N^\wedge$ is trivialized using $dx|_{\N^\wedge}$. We also write $x$ for this function and $\Z^\wedge_y$ for $\wp_\wedge^{-1}(y)$. The vector bundle $E$ lifts to $E^\wedge$ via $\pi_\N$. Let $\Phi_*$ be a push-forward map from sections of $E^\wedge$ over $\N^\wedge$ near $\N$ to sections of $E$ over $\M$ near $\N$ determined by trivializing geometric data as described in \cite[Sections 2 and 3]{GiKrMe10}. It is an isometry between the corresponding $x^{-1}L^2_b$-spaces on these neighborhoods, with adjoint given by pull-back $\Phi^*$. Let
\begin{equation}\label{kappaaction}
\kappa_{\varrho}u(z,x) = \varrho u(z,\varrho x),\ \varrho>0,
\end{equation}
where $\varrho x$ refers to the radial action on the fibers of $\N^\wedge$ and translation in $E^\wedge$ along fibers is the canonical one; $\kappa$ is a unitary map in $x^{-1}L^2_b(\Z^\wedge_y;E_{\Z^\wedge_y})$. 

The normal family of $A$ is defined by the formula 
$$
A_\wedge(dg)u = \lim_{\varrho \to \infty}\varrho^{-2}\kappa_{\varrho}^{-1}e^{-i\varrho\wp_\wedge^*g}\Phi^*A\Phi_*e^{i\varrho\wp_\wedge^*g}\kappa_{\varrho}u,
$$
(see \cite[Proposition~2.2]{GiKrMe10}) where $g \in C^{\infty}(\Y)$, $u \in C_c^{\infty}(\open \N^\wedge;E^\wedge)$. It satisfies
\begin{equation}\label{kappahomogeneity}
A_\wedge(\varrho\pmb\eta) = \varrho^2 \kappa_{\varrho}A_\wedge(\pmb\eta)\kappa_{\varrho}^{-1}
\end{equation}
for $\varrho > 0$ and $\pmb\eta\in T^*\Y\minus 0$. For each $\pmb\eta$, $A_\wedge(\pmb\eta)$ is an element of $x^{-2}\Diff^2_b(\Z_y^\wedge;E_{\Z_y^\wedge})$. Its indicial family is canonically identifiable with that of $A$ at $y=\pi_\Y(\pmb\eta)$, and the operators $\bA(y,D_x)=x^{-2}\bPhat(y,xD_x)$ are identifiable with $A_\wedge(\pmb 0_y)$.

\begin{lemma}\label{Symmetry}
Let $A \in x^{-2}\Diff^2_e(\M;E)$ act as an unbounded operator in $x^{-1}L^2_b$, and suppose that $A = A^{\star} \geq 1$ on $C_c^{\infty}(\open\M;E)$. Then the following holds.
\begin{enumerate}[(a)]
\item\label{Symmetry_a} $\wsym(A) = \wsym(A)^{\star}$, and $\spec(\wsym(A)) \subset \overline{{\mathbb R}}_+$ everywhere on $\wT^*\M$.
\item\label{Symmetry_b} The indicial family $\bPhat(y,\sigma)$ satisfies
\begin{equation*}
\bPhat(y,\sigma) = [\bPhat(y,\overline{\sigma})]^{\star} : C^{\infty}(\Z_y;E_{\Z_y}) \to C^{\infty}(\Z_y;E_{\Z_y})
\end{equation*}
for all $y \in \Y$ and $\sigma \in \C$.
\item\label{Symmetry_c} The normal family
\begin{equation*}
A_\wedge(\pmb\eta) : C_c^{\infty}(\open\Z_y^\wedge;E_{\Z_y^\wedge}) \subset x^{-1}L^2_b(\Z_y^\wedge;E_{\Z_y^\wedge}) \to x^{-1}L^2_b(\Z_y^\wedge;E_{\Z_y^\wedge}),
\end{equation*}
where $y = \pi_{\Y}\pmb\eta$, $\pi_{\Y} : T^*\Y \to \Y$, satisfies $A_\wedge(\pmb\eta) = A_\wedge^{\star}(\pmb\eta) \geq 0$ on $T^*\Y$ in the sense that
\begin{gather*}
\langle A_\wedge(\pmb\eta)u,v \rangle_{x^{-1}L^2_b(\Z_y^\wedge;E_{\Z_y^\wedge})} = \langle u,A_\wedge(\pmb\eta)v \rangle_{x^{-1}L^2_b(\Z_y^\wedge;E_{\Z_y^\wedge})}, \\
\langle A_\wedge(\pmb\eta)u,u \rangle_{x^{-1}L^2_b(\Z_y^\wedge;E_{\Z_y^\wedge})} \geq 0 
\end{gather*}
for all $u,v \in C_c^{\infty}(\open\Z_y^\wedge;E_{\Z_y^\wedge})$.
\end{enumerate}
\end{lemma}

\begin{proof}
Claim (\ref{Symmetry_b}) follows directly from the symmetry of $A$ in $x^{-1}L^2_b$ and does not make use of the semiboundedness. Claims (a) and (c) are consequences of the oscillatory tests used to define the $w$-principal symbol and the normal family from the action of the operator, in a fashion similar to the standard case, as follows.

To prove (\ref{Symmetry_a}), let $f\in \Ring$ and let $\phi \in C^{\infty}(\M;\wT^*\M)$ be such that $\wev^*\phi = df$. Let $\psi, \tilde \psi \in C_c^{\infty}(\open\M;E)$. Then
\begin{multline*}
\int_{\M}\langle \wsym(A)(\phi)\psi,\tilde{\psi}\rangle\, x^2d\m_b =
\lim_{\varrho \to \infty}\int_{\M}\langle \varrho^{-2}e^{-i\varrho f}A(e^{i\varrho f}\psi),\tilde{\psi}\rangle\, x^2d\m_b \\
= \lim_{\varrho \to \infty}\int_{\M}\langle \psi,\varrho^{-2}e^{-i\varrho f}A(e^{i\varrho f}\tilde{\psi})\rangle\, x^2d\m_b 
= \int_{\M}\langle \psi,\wsym(A)(\phi)\tilde{\psi}\rangle\, x^2d\m_b.
\end{multline*}
Since $\psi$, $\tilde{\psi}$, and $\phi$ are arbitrary, we conclude that $\wsym(A) = \wsym(A)^{\star}$ over $\open\M$, and by continuity then necessarily everywhere on $\wT^*\M$. Next,
\begin{multline*}
\int_{\M}\langle \wsym(A)(\phi)\psi,\psi\rangle\, x^2d\m_b =
\lim_{\varrho \to \infty}\int_{\M}\langle \varrho^{-2}e^{-i\varrho f}A(e^{i\varrho f}\psi),\psi\rangle\, x^2d\m_b \\
= \lim_{\varrho \to \infty}\varrho^{-2}\langle A(e^{i\varrho f}\psi),e^{i\varrho f}\psi\rangle_{x^{-1}L^2_b} \geq
\lim_{\varrho \to \infty}\varrho^{-2}\langle \psi,\psi \rangle_{x^{-1}L^2_b} = 0.
\end{multline*}
Since $\phi$ and $\psi$ are arbitrary this shows that $\wsym(A) \geq 0$ over $\open\M$ and again by continuity on all of $\M$.

To show (\ref{Symmetry_c}), let $g \in C^{\infty}(\Y)$ and  $u$, $v \in C_c^{\infty}(\open \N^\wedge;E^\wedge)$. Then 
\begin{align*}
\int_{\open \N^\wedge} &\langle A_\wedge(dg)u,v \rangle \,x^2d\m_b \\
&= \lim_{\varrho \to \infty}\int_{\open \N^\wedge}\langle \varrho^{-2}\kappa_{\varrho}^{-1}e^{-i\varrho\wp_\wedge^*g}\Phi^*A\Phi_*e^{i\varrho\wp_\wedge^*g}\kappa_{\varrho}u,v\rangle\,x^2d\m_b \\
&= \lim_{\varrho \to \infty}\int_{\M}\varrho^{-2}\langle A\Phi_*e^{i\varrho\wp_\wedge^*g}\kappa_{\varrho}u,
\Phi_*e^{i\varrho\wp_\wedge^*g}\kappa_{\varrho}v\rangle\,x^2d\m_b \\
&= \lim_{\varrho \to \infty}\int_{\open \N^\wedge}\langle u,\varrho^{-2}\kappa_{\varrho}^{-1}e^{-i\varrho\wp_\wedge^*g}\Phi^*A\Phi_*e^{i\varrho\wp_\wedge^*g}\kappa_{\varrho}v\rangle\,x^2d\m_b \\
&= \int_{\open \N^\wedge} \langle u,A_\wedge(dg)v \rangle \,x^2d\m_b.
\end{align*}
Because $g$, $u$, and $v$ are arbitrary, this shows that the normal family $A_\wedge(\pmb\eta)$ is symmetric on $C_c^{\infty}(\Z_y^\wedge;E_{\Z_y^\wedge})$. An argument analogous to the one used above for $\wsym(A)$ shows that $A_\wedge(\pmb\eta) \geq 0$ on $C_c^{\infty}(\Z_y^\wedge;E_{\Z_y^\wedge})$.
\end{proof}

Condition (\ref{w-ellipticity}), the invertibility of $\wsym(A)$ on $\wT^*\M\minus 0$, implies in view of (\ref{Symmetry_a}) of Lemma~\ref{Symmetry} that $\wsym(A) = \wsym(A)^{\star} > 0$ on $\wT^*\M\minus 0$. By homogeneity of the $w$-principal symbol and compactness of $\M$, there exists a constant $c_0 > 0$ such that
\begin{equation}\label{wsymLowerBound}
\wsym(A)(\pmb\xi) \geq c_0|\pmb\xi|^2
\end{equation}
for $\pmb\xi \in \wT^*\M$, where $|\cdot|$ corresponds to a metric on $\wT^*\M$ and $c_0>0$. In particular,
\begin{equation}\label{wsymparamell}
\wsym(A)(\pmb\xi) + \lambda^2 : \wpi^*E_{\pmb\xi} \to \wpi^*E_{\pmb\xi}, \quad \wpi : \wT^*\M \to \M,
\end{equation}
is invertible for all $(\pmb\xi,\lambda) \in \big(\wT^*\M\times\R\big)\minus 0$. This means that the operator $A + \lambda^2 \in x^{-2}\Diff_e^2(\M;E)$ is $w$-elliptic with parameter $\lambda \in \R$. This is an important property that will be used later in Section~\ref{sec-ExtensionOperator}.

The $w$-ellipticity of $A$ also implies that the indicial family $\bPhat$ consists of elliptic operators. The family
\begin{equation}\label{FredholmFamily}
\bPhat(y,\sigma) : H^2(\Z_y;E_{\Z_y}) \to L^2(\Z_y;E_{\Z_y})
\end{equation}
is a Fredholm family for each $y \in \Y$ that depends holomorphically on $\sigma \in \C$ and has a meromorphic inverse. As pointed out in Section~\ref{sec-Overview}, the indicial roots of $A$, the poles of $\bPhat(y,\sigma)^{-1}$, form a discrete subset in $\C$ for each $y \in \Y$ with only finitely many of these poles located in any given strip $|\Im(\sigma)|<r$. 

The indicial roots in general change with $y \in \Y$, and in principle they can vary all over the complex plane. By \cite{KrMe12a}, (\ref{b-strips}) guarantees first the existence of the vector bundle $\trb\to\Y$ associated with the Fredholm family \eqref{FredholmFamily} restricted to $\Sigma=\set{\sigma:-1<\Im\sigma<1}$ that was described in the previous section, next the existence of the splitting 
\begin{equation*}
\trb=\trb_F\oplus\trb_{aF}.
\end{equation*}

The normal family $A_\wedge(\pmb\eta)$ is for each $\pmb\eta \in T^*\Y$ an elliptic cone operator that is semibounded from below by Lemma~\ref{Symmetry}, and consequently has a Friedrichs extension. The Friedrichs extension for cone operators was systematically investigated in \cite{GiMe01}, and the main result there applied to $A_\wedge(\pmb\eta)$ gives that the Friedrichs domain is
\begin{equation*}
\Dom_{\wedge,F} = \Dom_{\wedge,\min} \oplus \omega \trb_{F,y},
\end{equation*}
where $\omega \in C_c^{\infty}(\overline{\R}_+)$ is a cut-off function near zero (i.e. $\omega \equiv 1$ near zero) that we consider as a function on $\Z_y^\wedge$ as usual ($\Dom_{\wedge,F}$ is independent of the choice of cut-off function in view of ellipticity). We leave implicit a reference to $\pmb\eta$ but point out that the minimal domain $\Dom_{\wedge,\min}$ of $A_\wedge(\pmb\eta)$ depends on $\pmb\eta \in T^*\Y$ only through $y = \pi_{\Y}\pmb\eta$ by (\ref{w-ellipticity}) and (\ref{b-strips}). These domains form a Hilbert space bundle over $\Y$ where the typical fiber is a cone Sobolev space, see \cite[Proposition~4.1]{GiKrMe10}. This implies that the Friedrichs domain $\Dom_{\wedge,F}$ of $A_\wedge(\pmb\eta)$ only depends on $y = \pi_{\Y}\pmb\eta$, and these domains also form a Hilbert space bundle over $\Y$.

Assumptions (\ref{w-ellipticity}) and (\ref{b-strips}) further imply that all closed extensions of $A_\wedge(\pmb\eta)$ are Fredholm for $\pmb\eta \in T^*\Y\minus 0$. In particular, the Friedrichs extension
$$
A_{\wedge,F}(\pmb\eta) : \Dom_{\wedge,F} \to x^{-1}L^2_b
$$
is Fredholm for $\pmb\eta \in T^*\Y\minus 0$ and selfadjoint in $x^{-1}L^2_b$ with $A_{\wedge,F}(\pmb\eta) \geq 0$ by Part (\ref{Symmetry_c}) of Lemma~\ref{Symmetry}. Assumption (\ref{AhatInvertible}) now guarantees that $A_{\wedge,F}(\pmb\eta)$ is in fact invertible, and using \eqref{kappahomogeneity} it is easy to see that this is equivalent to the existence of a constant $c_1 > 0$ such that
\begin{equation*}
A_\wedge(\pmb\eta) \geq c_1|\pmb\eta|^2 \textup{ on } C_c^{\infty}(\Z_y^\wedge;E_{\Z_y^\wedge}),
\end{equation*}
where $y = \pi_{\Y}\pmb\eta$, for all $\pmb\eta \in T^*\Y$. Here $|\cdot|$ is a metric on $T^*\Y$ (and the constant $c_1$ depends of course on the metric). In particular,
\begin{equation}\label{normalfamparamell}
A_\wedge(\pmb\eta) + \lambda^2 : \Dom_{\wedge,F} \to x^{-1}L^2_b(\Z_y^\wedge;E_{\Z^\wedge_y})
\end{equation}
is invertible for all $(\pmb\eta,\lambda) \in \bigl(T^*\Y\times\R\bigr)\minus 0$. This parameter-dependent ellipticity condition is the counterpart for the normal family to what \eqref{wsymparamell} is for the $w$-principal symbol. The invertibility of \eqref{normalfamparamell} on $\bigl(T^*\Y\times\R\bigr)\minus 0$ is another important property that we will use later in Section~\ref{sec-ExtensionOperator}.

We remark that, as an operator on $\Z_y^\wedge$,
\begin{equation*}
A_\wedge(\pmb\eta) : C_c^{\infty}(\open \Z_y^\wedge;E_{\Z^\wedge_y}) \subset x^{-1}L^2_b(\Z_y^\wedge;E_{\Z^\wedge_y}) \to x^{-1}L^2_b(\Z_y^\wedge;E_{\Z^\wedge_y})
\end{equation*}
decomposes as 
\begin{equation*}
A_\wedge(\pmb\eta) = A_\wedge(\pmb 0_y) + B(\pmb\eta) + C(\pmb\eta) 
\end{equation*}
where $C(\pmb\eta)$ acts as a bundle endomorphism on $E_{\Z^\wedge_y}$ as follows. Given $\nu\in \Z^\wedge_y$, let $p=\pi_\N(\nu)$ and let $\pmb\xi_{\pmb\eta,\nu}\in \wT_p^*\M$ be the unique element $\pmb\xi$ such that $\iota_\N^*\wev^*\pmb \xi=\wp^*\pmb\eta$ at $p$ and $\iota_{\Z^\wedge_y}^*\Phi^*\wev^*\pmb \xi=0$ with $\Phi$ the tubular neighborhood map used in the definition of $A_\wedge(\pmb\eta)$; $\iota_\N$ and $\iota_{\Z^\wedge_y}$ denote inclusion maps. Then
\begin{equation*}
C(\pmb\eta)=\wsym(A)(\pmb \xi_{\pmb\eta,\nu}).
\end{equation*}
We have
\begin{equation*}
A_\wedge(\pmb 0_y) + C(\pmb\eta) \geq c_0|\pmb\eta|^2 \text{ on } C_c^{\infty}(\Z_y^\wedge;E_{\Z^\wedge_y})
\end{equation*}
by (\ref{w-ellipticity}) and (\ref{b-strips}) (through \eqref{wsymLowerBound} and Part (\ref{Symmetry_c}) of Lemma~\ref{Symmetry}) for some $c_0 > 0$. Consequently, assumption (\ref{AhatInvertible}) on the normal family really only concerns the behavior of the first order term $B(\pmb\eta)$. In particular, in cases where $B(\pmb\eta) = 0$ (``product type operators''), assumption (\ref{AhatInvertible}) follows from (\ref{w-ellipticity}).


\section{The main result}\label{sec-MainResult}

Let $A \in x^{-2}\Diff^2_e(\M;E)$ be such that the standing assumptions (\ref{SymmetricSemibounded})-(\ref{AhatInvertible}) in Section~\ref{sec-SetupAssumptions} hold.

Define
\begin{equation}\label{SmoothFriedrichs}
\preP : C^{\infty}(\Y;\trb_F) \to C^{\infty}(\open\M;E)
\end{equation}
via $\preP(\tau) = \Phi_*(\omega\jmath\tau)$. Here $\omega \in C_c^{\infty}(\overline{\R}_+)$ is a fixed cut-off function supported near the origin with $\omega \equiv 1$ near $x = 0$, which we view as function on $\N^\wedge$ via the map $\N^\wedge\to\R$ determined by $x$, and $\Phi_*$ a push-forward map constructed using a tubular neighborhood map and parallel transport in $E$ along fibers as described in \cite{GiKrMe10}; it identifies sections of $E^\wedge$ on $\N^\wedge$ near the zero section of $\N^\wedge$ with sections of $E$ on $\M$ near the boundary $\N$. The map
$$
\jmath : C^{\infty}(\Y;\trb_F) \to C^{\infty}(\N^{\wedge};E^{\wedge})
$$
is defined fiberwise as the map that takes an element $\tau\in \trb_{F,y}$ and regards it as a section $\jmath_y\tau$ of $E^\wedge$ along $\Z^\wedge_y$ (which is what $\tau$ really is, see \eqref{TraceElement}).

Let
\begin{equation}\label{CinftyF}
C^{\infty}_{\trb_F}(\M;E) = \dot{C}^{\infty}(\M;E) + \preP(C^{\infty}(\Y;\trb_F)),
\end{equation}
where $\dot{C}^{\infty}(\M;E)$ denotes the $C^{\infty}$-sections of $E$ on $\M$ that vanish to infinite order on the boundary. 

\begin{theorem}\label{BasicFriedrichsTheorem}
The closure of
\begin{equation*}
A : C^{\infty}_{\trb_F}(\M;E) \subset x^{-1}L^2_b(\M;E) \to x^{-1}L^2_b(\M;E)
\end{equation*}
is the Friedrichs extension of $A$, i.e., the space $C^{\infty}_{\trb_F}(\M;E)$ is dense in $\Dom_F(A)$ with respect to the graph norm of $A$.
\end{theorem}

The proof is given below.

\medskip

\medskip
The operator $\preP$ can also be defined by the same formula on more general sections of $\trb_F$. However, this simple definition of an extension operator will not give that the image lies in $C^{\infty}(\open\M;E)$. In Section~\ref{sec-ExtensionOperator} we will construct a family of extension operators
\begin{equation*}
\Ext_{\lambda} : C^{-\infty}(\Y;\trb_F) \to C^{\infty}(\open\M;E),\ \lambda\in \R,
\end{equation*}
having the following properties:
\begin{enumerate}[\quad(a)]
\item\label{4.2a} $\Ext_{\lambda} - \preP : C^{\infty}(\Y;\trb_F) \to \dot{C}^{\infty}(\M;E)$;
\item\label{4.2b} there exists a $\delta_0 > 0$ such that $\Ext_{\lambda} : H^{2-\gen}(\Y;\trb_F) \to x^{\delta_0}H^{\infty}_e(\M;E)$;
\item\label{4.2c} for each $\lambda$, $\lambda'\in\R$, $\Ext_{\lambda} - \Ext_{\lambda'} : H^{2-\gen}(\Y;\trb_F) \to x^1H^{\infty}_e(\M;E)$;
\item\label{4.9d} $(A + \lambda^2) \circ \Ext_{\lambda} : H^{2-\gen}(\Y;\trb_F) \to x^{-1}H^{\infty}_e(\M;E)$ for all $\lambda\in \R$.
\end{enumerate}
Properties (\ref{4.2a})-(\ref{4.2c}) constitute Lemma~\ref{ExtBasicMapProps} and the last property is the content of Lemma~\ref{ExtAdvMapProps}. The Sobolev space $H^{2-\gen}(\Y;\trb_F)$ in (\ref{4.2b}) consists of sections of $\trb_F$ on $\Y$ of variable smoothness measured by the vector bundle endomorphism $2-\gen \in \End(\trb_F)$, where $\gen = 1 + x\partial_x$ as in \eqref{generator}. We refer to \cite{KrMe12b} for the definition and discussion of such spaces. 

Utilizing the Mellin transform it is immediate that \eqref{CinftyF} is a direct sum, and thus
\begin{equation}\label{iotaPrep}
\begin{bmatrix} \iota & \preP \end{bmatrix} :
\begin{array}{c} \dot{C}^{\infty}(\M;E) \\ \oplus \\ C^{\infty}(\Y;\trb_F) \end{array} \to C^{\infty}_{\trb_F}(\M;E)
\end{equation}
is bijective. Let 
\begin{equation*}
\Pi_{\Y} : \begin{array}{c} \dot{C}^{\infty}(\M;E) \\ \oplus \\ C^{\infty}(\Y;\trb_F) \end{array} \to C^{\infty}(\Y;\trb_F)
\end{equation*}
be the projection. The trace map,
\begin{equation*}
\gamma : C^{\infty}_{\trb_F}(\M;E) \to C^{\infty}(\Y;\trb_F),
\end{equation*}
is defined to be $\gamma=\Pi_\Y\circ \begin{bmatrix} \iota & \preP \end{bmatrix}^{-1}$. In other words,  $\gamma(u) = \tau$ if $u = u_0 + \preP(\tau) \in C^{\infty}_{\trb_F}(\M;E)$ with $u_0 \in \dot{C}^{\infty}(\M;E)$ and $\tau \in C^{\infty}(\Y;\trb_F)$.

\begin{theorem}\label{StructureFriedrichsDomain}
The operator $\Ext_\lambda$ restricts to a map
\begin{equation*}
\Ext_\lambda : H^{2-\gen}(\Y;\trb_F) \to \Dom_F(A)
\end{equation*}
and the map $\gamma : C^{\infty}_{\trb_F}(\M;E) \to C^{\infty}(\Y;\trb_F)$ extends by continuity to a bounded surjection
\begin{equation*}
\gamma : \Dom_F(A) \to H^{2-\gen}(\Y;\trb_F).
\end{equation*}
Let $\iota:x^1H^2_e(\M;E) \to \Dom_F(A)$ be the inclusion map.\footnote{Recall that $x^1H^2_e(\M;E)=\Dom_{\min}(A)$ by {\cite[Theorem 4.2]{GiKrMe10}}.} Then
\begin{equation}\label{iotaExtlambda}
\begin{bmatrix} \iota & \Ext_\lambda \end{bmatrix} : \begin{array}{c} x^1H^2_e(\M;E) \\ \oplus \\ H^{2-\gen}(\Y;\trb_F) \end{array} \to \Dom_{F}(A)
\end{equation}
is a topological isomorphism with inverse
\begin{equation}\label{InverseIotaExt}
\begin{bmatrix} \Id - (\Ext_\lambda\circ\gamma) \\ \gamma \end{bmatrix} : \Dom_F(A) \to \begin{array}{c} x^1H^2_e(\M;E) \\ \oplus \\ H^{2-\gen}(\Y;\trb_F) \end{array}.
\end{equation}
\end{theorem}

\begin{proof}[Proof of Theorem~\ref{BasicFriedrichsTheorem}]
Property (\ref{4.2a}) above and the bijectivity of the operator \eqref{iotaPrep} give that
\begin{equation}\label{iotaExtCinf}
\begin{bmatrix} \iota & \Ext_\lambda \end{bmatrix}  = \begin{bmatrix} \iota & \preP \end{bmatrix} \circ
\begin{bmatrix} \Id & \Ext_\lambda - \preP \\ 0 & \Id \end{bmatrix} :
\begin{array}{c} \dot{C}^{\infty}(\M;E) \\ \oplus \\ C^{\infty}(\Y;\trb_F) \end{array} \to C^{\infty}_{\trb_F}(\M;E)
\end{equation}
is also bijective. On the other hand, the map \eqref{iotaExtlambda} is a topological isomorphism which restricts to \eqref{iotaExtCinf}. Because the space of smooth functions $\dot{C}^{\infty}(\M;E) \oplus C^{\infty}(\Y;\trb_F)$ is dense in $x^1H^2_e(\M;E) \oplus H^{2-\gen}(\Y;\trb_F)$, we obtain that its image $C_{\trb_F}^{\infty}(\M;E)$ under the isomorphism \eqref{iotaExtlambda} is dense in $\Dom_F(A)$.
\end{proof}

\begin{proof}[Proof of Theorem~\ref{StructureFriedrichsDomain}]
Let $\scrH \hookrightarrow x^{-1}L^2_b(\M;E)$ denote the Hilbert space obtained by completion of $C_c^{\infty}(\open\M;E)$ with respect to the inner product
\begin{equation*}
[u,v]_\scrH = \langle Au,u \rangle_{x^{-1}L^2_b}, \ u,v \in C_c^{\infty}(\open\M;E).
\end{equation*}
We shall take advantage of the characterization of the Friedrichs domain as $\Dom_F(A) = \scrH\cap\Dom_{\max}(A)$.

We show first that $H^1_e(\M;E) \subset \scrH$ with continuous embedding. Indeed, the operator $A : H^1_e(\M;E) \to x^{-2}H^{-1}_e(\M;E)$ is continuous, and the $x^{-1}L^2_b$-inner product satisfies
$$
\bigl|\langle f,g \rangle_{x^{-1}L^2_b}\bigr| \leq C\|f\|_{x^{-2}H^{-1}_e}\,\|g\|_{H^{1}_e}
$$
for all $f,g \in C_c^{\infty}(\open\M;E)$ and some $C$. Thus, for $u \in C_c^{\infty}(\open\M;E)$, we have
$$
\bigl|\langle Au,u \rangle_{x^{-1}L^2_b}\bigr| \leq C \|Au\|_{x^{-2}H^{-1}_e}\,\|u\|_{H^{1}_e} \leq C' \|u\|^2_{H^1_e},
$$
for some $C'>0$, and of course $H^1_e(\M;E) \hookrightarrow x^{-1}L^2_b(\M;E)$. Consequently, $H^1_e(\M;E) \hookrightarrow \scrH$ as claimed. 

It follows that also $x^1H^2_e(\M\;E)$ is continuously embedded in $\scrH$. In view of (\ref{4.2b}) above we obtain
\begin{equation}\label{OpMatrixPre}
\begin{bmatrix} \iota & \Ext_{\lambda} \end{bmatrix} : \begin{array}{c} x^1H^2_e(\M;E) \\ \oplus \\ H^{2-\gen}(\Y;\trb_F) \end{array} \to \scrH.
\end{equation}
We now claim that the image of this map is also contained in $\Dom_{\max}(A)$, therefore is a subspace of $\Dom_F(A)$. Note that $(A+\lambda^2)$ maps $x^1H^2_e(\M;E)$ into $x^{-1}L^2_b(\M;E)$, while on the other hand (\ref{4.9d}) above (see Lemma~\ref{ExtAdvMapProps}) asserts that $(A+\lambda^2)\circ \Ext_\lambda$ maps $H^{2-\gen}(\Y;\trb_F)$ into $x^{-1}H^\infty_e(\M;E)$. Thus
\begin{equation}\label{OpMatrix}
\begin{bmatrix} A + \lambda^2 & (A + \lambda^2)\circ\Ext_{\lambda} \end{bmatrix} :
\begin{array}{c}
x^1H^2_e(\M;E) \\ \oplus \\ H^{2-\gen}(\Y;\trb_F)
\end{array}
\to x^{-1}L^2_b(\M;E)
\end{equation}
for every $\lambda \in \R$. Since $\begin{bmatrix} A + \lambda^2 & (A + \lambda^2)\circ\Ext_{\lambda} \end{bmatrix} =
(A + \lambda^2) \circ \begin{bmatrix} \iota & \Ext_{\lambda} \end{bmatrix}$, the image of \eqref{OpMatrixPre} is contained in $\Dom_{\max}(A)$. Continuity of the map \eqref{iotaExtlambda} (with the graph norm on $\Dom_F(A)$) follows from that of \eqref{OpMatrix}. We thus have that \eqref{iotaExtlambda} maps into $\Dom_F$ as claimed and we proceed to show that this map is an isomorphism.

By Proposition~\ref{Parametrix} the operator family \eqref{OpMatrix} is invertible for all $\lambda \in \R$ with sufficiently large $|\lambda|$. Fix one such $\lambda_0$ for which invertibility holds. Then
\begin{equation}\label{iotalambda0}
\begin{bmatrix} \iota & \Ext_{\lambda_0} \end{bmatrix} : \begin{array}{c} x^1H^2_e(\M;E) \\ \oplus \\ H^{2-\gen}(\Y;\trb_F) \end{array} \to \Dom_F(A)
\end{equation}
is invertible: Injectivity follows because $(A + \lambda_0^2)\circ \begin{bmatrix} \iota & \Ext_{\lambda_0} \end{bmatrix}$ is injective in view of the invertibility of \eqref{OpMatrix}. Let $R_{\lambda_0}$ be the range of \eqref{iotalambda0}. By construction of the Friedrichs extension,
$$
A + \lambda_0^2 : \Dom_F(A) \to x^{-1}L^2_b(\M;E)
$$
is invertible, and at the same time $A + \lambda_0^2 : R_{\lambda_0} \to x^{-1}L^2_b(\M;E)$ is invertible because \eqref{OpMatrix} is invertible. But $R_{\lambda_0}$ is a subspace of $\Dom_F(A)$, and therefore we necessarily have $R_{\lambda_0} = \Dom_F(A)$, which shows that \eqref{iotalambda0} is also surjective, hence invertible. 

The invertibility of \eqref{iotalambda0} now implies the invertibility of \eqref{iotaExtlambda} for arbitrary $\lambda$,
because
\begin{equation*}
\begin{bmatrix} \iota & \Ext_\lambda \end{bmatrix} = \begin{bmatrix} \iota & \Ext_{\lambda_0} \end{bmatrix}\circ\begin{bmatrix} \Id & \Ext_{\lambda} - \Ext_{\lambda_0} \\ 0 & \Id \end{bmatrix}
\end{equation*}
and
\begin{equation*}
\begin{bmatrix} \Id & \Ext_{\lambda} - \Ext_{\lambda_0} \\ 0 & \Id \end{bmatrix} : \begin{array}{c} x^1H^2_e(\M;E) \\ \oplus \\ H^{2-\gen}(\Y;\trb_F) \end{array} \to \begin{array}{c} x^1H^2_e(\M;E) \\ \oplus \\ H^{2-\gen}(\Y;\trb_F) \end{array}
\end{equation*}
is trivially invertible. Note that $\Ext_\lambda - \Ext_{\lambda_0} : H^{2-\gen}(\Y;\trb_F) \to x^1H^2_e(\M;E)$ by (\ref{4.2c}) above (or see Lemma \ref{ExtBasicMapProps}).

The factorization of $\begin{bmatrix} \iota & \Ext_\lambda \end{bmatrix}$ in \eqref{iotaExtCinf} yields $\Pi_{\Y} \circ \begin{bmatrix} \iota & \Ext_\lambda \end{bmatrix}^{-1} = \Pi_{\Y} \circ \begin{bmatrix} \iota & \preP \end{bmatrix}^{-1}$, which is $\gamma$ by definition. By the invertibility of \eqref{iotaExtlambda},
\begin{equation}\label{InverseOfSplit}
\begin{bmatrix} \iota & \Ext_\lambda \end{bmatrix}^{-1} : \Dom_F(A) \to \begin{array}{c} x^1H^2_e(\M;E) \\ \oplus \\ H^{2-\gen}(\Y;\trb_F) \end{array}
\end{equation}
is continuous, and trivially $\Pi_{\Y}$ extends continuously to the projection onto the space $H^{2-\gen}(\Y;\trb_F)$. Consequently $\Pi_{\Y} \circ \begin{bmatrix} \iota & \Ext_\lambda \end{bmatrix}^{-1}$, now as a map
$$
 \Dom_F(A) \to H^{2-\gen}(\Y;\trb_F),
$$
is the stated continuous extension of $\gamma$. 
\end{proof}

The bijectivity of \eqref{iotaPrep} gives that 
\begin{equation*}
0\to \dot C^\infty(\M;E) \xrightarrow{\iota}C^\infty_{\trb_F}(\M;E)\xrightarrow{\gamma}C^\infty(\Y;\trb_F)\to 0
\end{equation*}
is an exact sequence. Extending this we have

\begin{theorem}\label{ExactSequenceFriedrichsTheorem}
The short sequence
\begin{equation*}
0\to x^1H^2_e(\M;E) \xrightarrow{\iota}\Dom_F(A)\xrightarrow{\gamma}H^{2-\gen}(\Y;\trb_F)\to 0
\end{equation*}
is exact. 
\end{theorem}

In the following we take $\lambda=0$ and write $\Ext$ for $\Ext_0$. 

\begin{proof} Writing $\begin{bmatrix} \alpha \\ \gamma \end{bmatrix}$ for the map \eqref{InverseOfSplit} we obtain from 
\begin{equation*}
\begin{bmatrix} \iota & \Ext \end{bmatrix} \circ \begin{bmatrix} \alpha \\ \gamma \end{bmatrix} = \alpha + \Ext\circ\gamma = \Id : \Dom_F(A) \to \Dom_F(A),
\end{equation*}
that  $\alpha = \Id-\Ext\circ \gamma$ maps into $x^1H^2_e(\M;E)$ and also that the inverse of \eqref{iotaExtlambda} is indeed \eqref{InverseIotaExt}. In particular, $u\in \Dom_F(A)$ and $\gamma(u)=0$ implies $u\in x^1H^2_e(\M;E)$. From
\begin{equation*}
\begin{bmatrix} \alpha \\ \gamma \end{bmatrix} \circ \begin{bmatrix} \iota & \Ext \end{bmatrix} = \begin{bmatrix} \alpha\circ\iota & \alpha\circ\Ext \\ \gamma\circ\iota & \gamma\circ\Ext \end{bmatrix} = \begin{bmatrix} \Id & 0 \\ 0 & \Id \end{bmatrix} : \begin{array}{c} x^1H^2_e(\M;E) \\ \oplus \\ H^{2-\gen}(\Y;\trb_F) \end{array} \to \begin{array}{c} x^1H^2_e(\M;E) \\ \oplus \\ H^{2-\gen}(\Y;\trb_F) \end{array}
\end{equation*}
we get $\gamma\circ\iota = 0$, which gives exactness at $\Dom_F(A)$. This shows that $\ker(\gamma) = x^1H^2_e(\M;E)$ as stated, and, moreover, it implies that $\gamma\circ\Ext = \Id$ on $H^{2-\gen}(\Y;\trb_F)$ which shows that $\gamma : \Dom_F(A) \to H^{2-\gen}(\Y;\trb_F)$ is surjective with right inverse $\Ext$. 
\end{proof}

\begin{theorem}\label{FriedrichsIsFredholm}
There is $\delta_0 > 0$ such that $\Dom_F(A) \hookrightarrow x^{\delta_0}H^2_e(\M;E)$. Consequently, the Friedrichs extension $A : \Dom_F(A) \to x^{-1}L^2(\M;E)$ is Fredholm with compact resolvent. 
\end{theorem}
\begin{proof} By (\ref{4.2b}) above we have
$$
\Ext : H^{2-\gen}(\Y;\trb_F) \to x^{\delta_0}H^{\infty}_e(\M;E)
$$
for some $0 < \delta_0 < 1$, and therefore
$$
\begin{bmatrix} \iota & \Ext \end{bmatrix} : \begin{array}{c} x^1H^2_e(\M;E) \\ \oplus \\ H^{2-\gen}(\Y;\trb_F) \end{array} \to x^1H^2_e(\M;E) + x^{\delta_0}H^{\infty}_e(\M;E) \hookrightarrow x^{\delta_0}H^2_e(\M;E).
$$
This shows that $\Dom_F(A) \hookrightarrow x^{\delta_0}H^2_e(\M;E)$, and because the latter edge Sobolev space embeds compactly into $x^{-1}L^2_b(\M;E)$ we obtain that $A : \Dom_F(A) \to x^{-1}L^2_b(\M;E)$ is Fredholm with compact resolvent as claimed. 
\end{proof}


\section{The extension operator}\label{sec-ExtensionOperator}

We continue to work with an operator $A \in x^{-2}\Diff^2_e(\M;E)$ under the  standing assumptions (\ref{SymmetricSemibounded})-(\ref{AhatInvertible}) stated in Section~\ref{sec-SetupAssumptions}.

Fix an arbitrary Riemannian metric $g_{\Y}$ on $\Y$ and a positive number $\delta \ll 1$. Fix also a $C^{\infty}$ function $[\cdot] : \R \to \R$ such that $[r] \geq 1$ for all $r \in \R$ and $[r] = |r|$ for all $|r| \geq r_0$ for some sufficiently large $r_0 > 0$, and a cut-off function $\omega \in C_c^{\infty}(\overline{\R}_+)$ with support contained in $\lbra 0,c_0\rpar$ for some $0 < c_0 \ll 1$ such that $\omega \equiv 1$ near $x = 0$.

Let $U \subset \Y$ be an open set, small enough so that $\trb_F|_U$ allows for a $\delta$-admissible trivialization over $U$ with respect to the endomorphism $\gen \in \End(\trb_F)$ in \eqref{generator}. This means (see \cite{KrMe12b}) that the restriction $\trb_F|_U$ splits into a direct sum of trivial $\gen$-invariant subbundles $\trb_F|_U \cong \bigoplus_k\trb^k_{F,U}$ such that there exist closed pairwise disjoint disks $D_k$ in the complex plane, of diameter $< \delta$, such that, for each $k$, the spectrum of the restriction of $\gen$ to each fiber of $\trb^k_{F,U}$ over $U$ is contained in $D_k$. We refer to such a trivialization as a $\delta$-admissible trivialization. 

A section $\tau$ of $\trb_U$ is of course an element $\tau(y) \in \trb_{F,y}$ for each $y\in U$, that is, a singular function on $\Z^\wedge_y$ of the form in \eqref{TraceElement}. Smoothness of $\tau$ means that it, regarded as a function on $\open\wp_\wedge^{-1}(U) \subset \open\N^\wedge$, is a $C^{\infty}$ section of $\wp_\wedge^*E$. By $\open\wp_\wedge$ we mean $\wp_\wedge$ restricted to the interior of $\N^\wedge$. We refer to \cite{KrMe12a} for further information about vector bundles like $\trb_F$.

Pick a smooth local frame $\tau_{\mu} : U \to \trb_F$, $\mu = 1,\ldots,N$, respecting the $\delta$-admissible trivialization described above. We refer to such a local frame as a $\delta$-admissible frame.  Assuming that $U$ is small enough, we have that $\open\wp_\wedge^{-1}(U) \cong U \times \open\Z^\wedge$ with the typical fiber $\Z$ of the boundary fibration $\wp : \N \to \Y$, and so each $\tau_{\mu}$ can equally be considered as defined and smooth on $U \times \Z\times(0,\infty)$. Let $\Z^\wedge$ be the typical fiber of $\wp_\wedge:\N^\wedge\to \Y$.

Suppose further that $U$ is the domain of local coordinates $y_1,\ldots,y_q$ for $\Y$. Suppose that the metric on $T^*\Y$ takes the form
\begin{equation*}
|\eta|_y^2 = \sum_{k,\ell} g^{k \ell}(y)\eta_k\eta_{\ell}
\end{equation*}
in these coordinates. Define $[\eta,\lambda]_y = \bigl[ |\eta|_y^2 + \lambda^2 \bigr]$.

For arbitrary $\varphi,\psi \in C_c^{\infty}(U)$ and each $\lambda \in \R$ define
${}_{\psi}(\ext_{\lambda})_{\varphi} : C^{-\infty}(U) \to C^{\infty}(U\times\open\Z^\wedge)$ by
$$
[{}_{\psi}(\ext_{\lambda})_{\varphi}(f)](y,z,x) = \frac{\psi(y)}{(2\pi)^q}\int_{\R^q} e^{iy\eta}\omega(x[\eta,\lambda]_y)(\varphi f)\ft(\eta)\,d\eta,
$$
where $(\varphi f)\ft$ is the Fourier transform of $\varphi f$. Observe that $[{}_{\psi}(\ext_{\lambda})_{\varphi}(f)](y,z,x)$ is independent of $z \in \Z$, but we consider it a smooth function on $U\times\open\Z^\wedge$, or $\open\wp_\wedge^{-1}(U)$. 

Next, using the $\delta$-admissible frame introduced above, define
\begin{equation*}
{}_{\psi}(\Ext_{\lambda})_{\varphi} : C^{-\infty}(U;\trb_F) \to C^{\infty}(U\times\open\Z^\wedge;E^\wedge)
\end{equation*} 
by
\begin{equation*}
{}_{\psi}(\Ext_{\lambda})_{\varphi}(\tau) = \sum\limits_{\mu=1}^N {}_{\psi}(\ext_{\lambda})_{\varphi}(u^{\mu})\tau_{\mu} \textup{ if }
\tau = \sum\limits_{\mu=1}^N u^{\mu}\tau_{\mu}.
\end{equation*}
For every $\tau \in C^{-\infty}(U;\trb_F)$, the function ${}_{\psi}(\Ext_{\lambda})_{\varphi}(\tau)$ is a $C^{\infty}$ function on $\open\wp_\wedge^{-1}(U)$. Its support is contained in a compact subset of $\wp_\wedge^{-1}(U)$ near the boundary $\wp^{-1}(U)$, and thus can be considered a smooth section of $E^\wedge$ on $\open\N^\wedge$ that is supported near $\N$.

Now let $U_s$, $s=1,\ldots, S$, be a finite open covering of $\Y$ with open sets $U_s$ of the kind used above, and let $\{\varphi_s\}_{s=1}^S$ be a subordinate partition of unity. Moreover, let $\psi_s \in C_c^{\infty}(U_s)$ be such that $\psi_s \equiv 1$ in a neighborhood of $\supp \varphi_s$. On each of the sets $U_s$ we have an operator ${}_{\psi_s}(\Ext_{\lambda})_{\varphi_s}$ as just defined. 

Define 
\begin{equation*}
\Ext_{\lambda} = \Phi_*\sum_{s=1}^S {}_{\psi_s}(\Ext_{\lambda})_{\varphi_s} : C^{-\infty}(\Y;\trb_F) \to C^{\infty}(\open\M;E).
\end{equation*}
We shall refer to $\Ext=\Ext_0$ simply as the extension operator.

Recall that $\Phi_*$ is the push-forward map that identifies sections of $E^\wedge$ on $\N^\wedge$ with sections of $E$ on $\M$ near the boundary $\N$. We utilized a similar construction of an extension operator acting from sections of the full trace bundle associated with a first order elliptic wedge operator into the maximal domain of that operator in \cite[Section~7]{KrMe13}. We will follow the ideas there to establish the mapping properties of the family of extension operators, thereby taking advantage of certain technical improvements obtained in \cite{Kr15}.

\begin{lemma}\label{ExtBasicMapProps}
The family of extension operators $\Ext_{\lambda} : C^{-\infty}(\Y;\trb_F) \to C^{\infty}(\open\M;E)$ has the following mapping properties:
\begin{enumerate}[(a)]
\item With the map $\preP : C^{\infty}(\Y;\trb_F) \to C^{\infty}(\open\M;E)$ in \eqref{SmoothFriedrichs} we have
$$
\Ext_{\lambda} - \preP : C^{\infty}(\Y;\trb_F) \to \dot{C}^{\infty}(\M;E).
$$
\item There exists $\delta_0 > 0$ such that
$$
\Ext_{\lambda} : H^{2-\gen}(\Y;\trb_F) \to x^{\delta_0}H^{\infty}_e(\M;E).
$$
\item We have
$$
\Ext_{\lambda} - \Ext_{\lambda'} : H^{2-\gen}(\Y;\trb_F) \to x^1H^{\infty}_e(\M;E).
$$
\end{enumerate}
\end{lemma}
\begin{proof}
We prove these claims by a careful analysis of the operators
$$
{}_{\psi}(\Ext_{\lambda})_{\varphi} : C^{-\infty}(U;\trb_F) \to C^{\infty}(U\times\open\Z^\wedge;E^\wedge)
$$
that appear in the definition of $\Ext_{\lambda}$.

The first claim follows from a direct analysis of the Fourier integrals that are involved in the definition of the operators ${}_{\psi}(\ext_{\lambda})_{\varphi}$ that act on the coefficients of the section $\tau \in C^{\infty}(U;\trb_F)$ with respect to the chosen $\delta$-admissible frame $\{\tau_{\mu}\}_{\mu=1}^N$ of $\trb_F$. We have
$$
{}_{\psi}(\ext_{\lambda})_{\varphi}(f) - \psi\varphi f \in \dot{C}^{\infty}
$$
for every $f \in C^{\infty}(U)$ (see \cite[Lemma~7.6]{KrMe13} for a proof). This implies that
$$
[{}_{\psi}(\Ext_{\lambda})_{\varphi}](\tau) - \psi(y)\varphi(y)\omega(x)\jmath_y\tau(y,z,x) \in \dot{C}^{\infty}(U\times\Z^\wedge;E^\wedge)
$$
for $\tau \in C^{\infty}(U;\trb_F)$, from which (a) follows immediately. Recall that $\jmath_y$ is the map that takes an element $\tau\in \trb_{F,y}$ and regards it as a section $\jmath_y\tau$ of $E^\wedge$ along $\Z^\wedge\approx \Z^\wedge_y$.

Before proceeding with the remainder of the proof we remind the reader of the (weighted) cone Sobolev spaces\footnote{For readers familiar with Melrose's scattering pseudodifferential calculus \cite{Mel95}, we note that the change of variables $x' = 1/x$ yields an isomorphism between distributions in the space $H^s_{\textup{cone}}$ supported away from $x = 0$, as they occur in the definition of $\K^{s,\nu}_t$, and distributions in the scattering Sobolev space $H^s_\sc(\Z\times\lbra0,1\rpar)$ supported near $x' = 0$.}
\begin{equation*}
\K^{s,\nu}_t(\Z^\wedge;E^\wedge) = \omega x^{\nu}H^s_b(\Z^\wedge;E^\wedge) + (1-\omega)x^{-t}H^s_{\textup{cone}}(\Z^\wedge;E^\wedge)
\end{equation*}
introduced by Schulze and discussed in his monographs on pseudodifferential operators on singular manifolds such as \cite{SchuNH,SchuWiley}. We will often just write $\K^{s,\nu}_t$ instead of using the full notation.

To prove (b) and (c) we shall take the point of view that
\begin{equation*}
{}_{\psi}(\Ext_{\lambda})_{\varphi} : C^{\infty}(U;\trb_F) \to C^{\infty}(U;C^{\infty}(\open\Z^\wedge;E^\wedge))
\end{equation*}
is a pseudodifferential operator acting from sections of the (finite-rank) bundle $\trb_F$ over $U$ to sections of an infinite-dimensional bundle over $U$ whose fiber is a function space on $\Z^\wedge$.

Disregarding the multiplication by $\varphi$ for the moment, define the symbol
\begin{equation*}
\sym(y,\eta,\lambda) : \trb_{F,y} \to C^{\infty}(\open\Z^\wedge;E^\wedge)
\end{equation*}
by
\begin{equation*}
\sym(y,\eta,\lambda)=\psi(y)\omega(x[\eta,\lambda]_y)\jmath_y.
\end{equation*}
If $\tau\in C^\infty(U;\trb_F)$ is a section, then 
\begin{equation}\label{symExtdef}
\sym(y,\eta,\lambda)[\tau(y)] = \big[(z,x) \mapsto \psi(y)\omega(x[\eta,\lambda]_y)\jmath_y\,\tau(y,z,x)\big]
\end{equation}
and we have, for some sufficiently small $\delta_0 > 0$,
\begin{equation*}
\sym(y,\eta,\lambda)[\tau(y)]  \in C^{\infty}(U\times\R^{q+1},\K^{\infty,\delta_0}_{\infty}(\Z^\wedge;E^\wedge))
\end{equation*}
for $y\in U$ and $(\eta,\lambda)\in \R^{q+1}$ on account of the nature of the elements of $\trb_F$ and the concept of smoothness for this bundle (see \cite{KrMe12a}). For the number $\delta_0$, recall the absence of boundary spectrum on the line $\Im\sigma=0$ (assumption (\ref{b-strips}) in Section~\ref{sec-SetupAssumptions}); any $0 < \delta_0 < 1$ that satisfies
$$
0 < \delta_0 < \inf\set{|\Im(\sigma)| : \Im(\sigma) < 0 \text{ and } \exists y \in \Y, (y,\sigma) \in \spec_e(A)}
$$
will work.

For any $\varrho > 0$, pull-back with respect to the diffeomorphism $\open\Z^\wedge \to \open\Z^\wedge$, $(z,x) \mapsto (z,\varrho x)$, and normalization induces the action
$$
\kappa_{\varrho}u(z,x) = \varrho u(z,\varrho x), \quad \varrho > 0,
$$
on $C^{-\infty}(\open\Z^\wedge;E^\wedge)$, see \eqref{kappaaction}. The action $\kappa_{\varrho}$ restricts to the spaces $\K^{s,\nu}_t(\Z^\wedge;E^\wedge)$ for all $s,\nu,t \in \R$, and it is a group of isometries on the reference Hilbert space $x^{-1}L^2_b(\Z^\wedge;E^\wedge)$. It also restricts to the fibers $\trb_{F,y}$ where we have $\kappa_{\varrho} = \varrho^{\gen} \in \End(\trb_F)$, see \eqref{generator}. Directly from \eqref{symExtdef} we see that
\begin{equation}\label{kappahomsym}
\sym(y,\varrho\eta,\varrho\lambda) = \kappa_{\varrho}\sym(y,\eta,\lambda)\kappa_{\varrho}^{-1} : \trb_{F,y} \to C^{\infty}(\open\Z^\wedge;E^\wedge)
\end{equation}
for all sufficiently large $|\eta|_y^2 + \lambda ^2$ and all $\varrho \geq 1$. This twisted homogeneity relation is analogous to \eqref{kappahomogeneity} for the normal family of $A$.

With the fixed $\delta$-admissible frame $\{\tau_{\mu}\}_{\mu=1}^N$ of $\trb_F$ over $U$ we quantize the symbol $\sym(y,\eta,\lambda)$ and get an operator family
\begin{equation}\label{localsigmaop}
\begin{gathered}
\sym(y,D_y,\lambda) : C_c^{\infty}(U;\trb_F) \to C^{\infty}(U;C^{\infty}(\open\Z^\wedge;E^\wedge)) \\
[\sym(y,D_y,\lambda)\tau](y) = \frac{1}{(2\pi)^q}\int_{\R^q}e^{iy\eta}\sum_{\mu=1}^N \sym(y,\eta,\lambda)[\tau_{\mu}(y)]\, \widehat{u^{\mu}}(\eta) \,d\eta
\end{gathered}
\end{equation}
for $\tau = \sum\limits_{\mu=1}^N u^{\mu}\tau_{\mu} \in C_c^{\infty}(U;\trb_F)$; observe that $[{}_{\psi}(\Ext_{\lambda})_{\varphi}](\tau) = \sym(y,D_y,\lambda)(\varphi \tau)$ for such $\tau$. The local symbol
\begin{equation}\label{localsigmasym}
\begin{gathered}
\sigma(y,\eta,\lambda) : \C^N \to C^{\infty}(\open\Z^\wedge;E^\wedge) \\
\sigma(y,\eta,\lambda)\begin{bmatrix} c^1 \\ \vdots \\ c^N \end{bmatrix} = \sum\limits_{\mu=1}^N \sym(y,\eta,\lambda)[\tau_{\mu}(y)]c^{\mu}
\end{gathered}
\end{equation}
that is associated to $\sym(y,\eta,\lambda)$ and the frame $\{\tau_{\mu}\}_{\mu=1}^N$ is $C^{\infty}$ in all variables taking values in $\L(\C^N,\K^{s,\delta_0}_t)$ for all $s,t \in \R$. It satisfies, for all $\alpha \in {\mathbb N}_0^q$ and $\beta \in {\mathbb N}_0^{q+1}$, an estimate
\begin{equation}\label{symbest}
\|\kappa_{\langle \eta,\lambda \rangle}^{-1} \bigl(D^{\alpha}_y\partial_{(\eta,\lambda)}^{\beta}\sigma(y,\eta,\lambda)\bigr) \langle \eta,\lambda \rangle^{g(y)}\|_{\L(\C^N,\K^{s,\delta_0}_t)} \lesssim \langle \eta,\lambda \rangle^{-|\beta|+\delta|\alpha|}
\end{equation}
for all $(\eta,\lambda) \in \R^{q+1}$, locally uniformly in $y \in U$. Here $g(y) \in C^{\infty}(U,\L(\C^N))$ is the local representation of the endomorphism $\gen \in \End(\trb_F)$ over $U$ with respect to the frame $\{\tau_{\mu}\}_{\mu=1}^N$, so $\langle \eta,\lambda \rangle^{g(y)} \in \L(\C^N)$ is the local representation of $\kappa_{\langle \eta,\lambda\rangle} \in \L(\trb_{F,y})$ with respect to that frame (at $y$). 

It is at this juncture where $\delta$-admissibility of the frame $\{\tau_{\mu}\}_{\mu=1}^N$ comes into play. It guarantees that we locally work with symbols of H{\"o}rmander type $1,\delta$ for some $0 < \delta \ll 1$ (which we fixed at the very beginning). The symbol estimates \eqref{symbest} follow from \eqref{kappahomsym} by differentiation as in \cite[Proposition~3.11]{KrMe12b}. The argument involves estimates on the $y$-derivatives of the local representation $\varrho^{-g(y)}$ of $\kappa_{\varrho}^{-1}$ that appears on the right in \eqref{kappahomsym} which rely on $\delta$-admissibility in a crucial manner.

Pseudodifferential operators with symbols that satisfy estimates of the kind \eqref{symbest} are continuous in
$$
H_{\textup{comp}}^{s'-g}(U;\C^N) \to {\mathcal W}_{\loc}^{s'}(U,\K^{s,\delta_0}_{t})
$$
for all $s',s,t \in \R$. This follows from untwisting the group action on the finite-rank part utilizing the calculus in \cite{KrMe12b}, combined with a boundedness result for abstract edge pseudodifferential operators with operator-valued symbols of H{\"o}rmander type $1,\delta$ in abstract wedge Sobolev spaces\footnote{For a Hilbert space $E$ equipped with a strongly continuous group action $\{\kappa_{\varrho}\}_{\varrho > 0}$ the abstract edge Sobolev space ${\mathcal W}^{s'}(\R^q,E)$ is the completion of ${\mathscr S}(\R^q,E)$ with respect to
$$
\|u\|_{{\mathcal W}^{s'}}^2 = \int_{\R^q} \langle \eta \rangle^{2s'}\|\kappa^{-1}_{\langle \eta \rangle}\widehat u(\eta)\|_{E}^2\,d\eta.
$$
This is a subspace of ${\mathscr S}'(\R^q,E)$. These spaces were introduced by Schulze, see \cite{SchuNH,SchuWiley} for further information.} proved in \cite[Theorem~4]{Kr15}. Now choose specifically $s' = \delta_0 + 1$, $s$ arbitrarily, and $t = s-\delta_0-\frac{\dim\Z^\wedge}{2}$. Then, as discussed in \cite[Appendix~A]{KrMe13},  the space ${\mathcal W}_{\textup{loc}}^{s'}(U,\K^{s,\delta_0}_t)$ is a local model space for the edge Sobolev space $x^{\delta_0}H^{s}_e(\M;E)$ near the boundary. The group action in \cite[Appendix~A]{KrMe13} was normalized differently, so our choice of $s' = \delta_0 + 1$ here corresponds to the choice $\delta_0 + 1/2$ there. Putting things together gives
$$
\Phi_*\bigl[{}_{\psi}(\Ext_{\lambda})_{\varphi}\bigr] : H^{1 + \delta_0 - \gen}(\Y;\trb_F) \to x^{\delta_0}H^{\infty}_e(\M;E).
$$
Because $\delta_0 < 1$ we have $H^{2-\gen}(\Y;\trb_F) \hookrightarrow H^{1+\delta_0-\gen}(\Y;\trb_F)$, and (b) is proved.

To prove (c) we may assume $\lambda'=0$ without loss of generality. With fixed $\lambda$ let  $c(y,\eta) = \sigma(y,\eta,\lambda) - \sigma(y,\eta,0)$. Note that
\begin{equation}\label{csymdef}
c(y,\eta) = \lambda\int_0^1 (\partial_{\lambda}\sigma)(y,\eta,\theta\lambda)\,d\theta,
\end{equation}
and that
\begin{equation*}
(\partial_{\lambda}\sigma)(y,\eta,\lambda) : \C^N \to C^{\infty}(\open\Z^\wedge;E^\wedge)
\end{equation*}
is given by
\begin{equation*}
(\partial_{\lambda}\sigma)(y,\eta,\lambda)\begin{bmatrix} c^1 \\ \vdots \\ c^N \end{bmatrix} =
\Bigl[(z,x) \mapsto \frac{\partial_{\lambda}[\eta,\lambda]_y}{[\eta,\lambda]_y}\sum\limits_{\mu=1}^N \psi(y)(x\partial_x\omega)(x[\eta,\lambda]_y)\jmath_y\tau_{\mu}(y,z,x)c^{\mu}\Bigr].
\end{equation*}
Observe that $x\partial_x\omega \in C_c^{\infty}(\R_+)$ is supported away from $x = 0$. Consequently, the symbol estimates \eqref{symbest} are better for $\partial_\lambda\sigma$ than what is stated there in the sense that the $\L(\C^N,\K^{s,\delta_0}_t)$-norms can be replaced by any $\L(\C^N,\K^{s,\nu}_t)$-norm for the full range of regularity parameters $s,\nu,t \in \R$. This implies that the symbol $c(y,\eta)$ satisfies the estimate
$$
\|\kappa_{\langle \eta \rangle}^{-1} \bigl(D^{\alpha}_y\partial_{\eta}^{\beta}c(y,\eta)\bigr) \langle \eta \rangle^{g(y)}\|_{\L(\C^N,\K^{s,\nu}_t)} \lesssim \langle \eta \rangle^{-1-|\beta|+\delta|\alpha|}
$$
for all $\alpha,\beta \in {\mathbb N}_0^q$  and all $\eta \in \R^q$, locally uniformly in $y \in U$, which yields as before that
$$
c(y,D_y) : H^{s'-g}_{\textup{comp}}(U;\C^N) \to {\mathcal W}_{\textup{loc}}^{s'+1}(U,\K^{s,\nu}_t)
$$
is continuous. For $s'=\nu=1$, $s$ arbitrary, and $t = s-\nu-\frac{\dim\Z^\wedge}{2}$ we have that ${\mathcal W}_{\textup{loc}}^{s'+1}(U,\K^{s,\nu}_t)$ is a local model space for $x^1H^s_e(\M;E)$ near the boundary. In view of \eqref{csymdef} this implies that
$$
\Phi_*\bigl[{}_{\psi}(\Ext_{\lambda})_{\varphi} - {}_{\psi}\Ext_{\varphi}\bigr] : H^{1-\gen}(\Y;\trb_F) \to x^1H^{\infty}_e(\M;E)
$$
which completes the proof of (c).
\end{proof}

\begin{lemma}\label{ExtAdvMapProps}
We have
$$
(A + \lambda^2) \circ \Ext_{\lambda} : H^{2-\gen}(\Y;\trb_F) \to x^{-1}H^{\infty}_e(\M;E)
$$
for all $\lambda \in \R$.
\end{lemma}
\begin{proof}
We prove the lemma by carefully investigating the structure of the composition
\begin{equation}\label{AlExt}
(A + \lambda^2)\circ {}_{\psi}(\Ext_{\lambda})_{\varphi} : C^{-\infty}(U;\trb_F) \to C^{\infty}(U\times\open\Z^\wedge;E^\wedge).
\end{equation}
This composition is to be interpreted as follows. The functions in the range of the operator $\sym(y,D_y,\lambda)$ in \eqref{localsigmaop} all have support contained in a fixed set $K\times\Z\times[0,c_0]$ with some compact set $K \Subset U$ (depending on the support of $\psi$) and some small $c_0 > 0$ (depending on the support of the cut-off function $\omega$). We identify sections of $E$ on $\M$ that are supported near the boundary $\N$ with sections of $E^\wedge$ on $\N^\wedge$ using pull-back and push-forward with respect to $\Phi$, and in this sense the composition \eqref{AlExt} is to be understood. 

The proof relies on elaborating further on the point of view taken in the proof of Lemma~\ref{ExtBasicMapProps} that the operator
$$
{}_{\psi}(\Ext_{\lambda})_{\varphi} : C^{\infty}(U;\trb_F) \to C^{\infty}(U;C^{\infty}(\open\Z^\wedge;E^\wedge))
$$
is a pseudodifferential operator with an operator-valued symbol acting from sections of the (finite-rank) bundle $\trb_F$ over $U$ to sections of an infinite-dimensional bundle over $U$ whose fiber is a function space on $\Z^\wedge$. Recall that
$$
[{}_{\psi}(\Ext_{\lambda})_{\varphi}](\tau) = \sym(y,D_y,\lambda)(\varphi \tau)
$$
for $\tau \in C^{\infty}(U;\trb_F)$ with the operator $\sym(y,D_y,\lambda)$ in \eqref{localsigmaop}. Its local symbol with respect to the fixed $\delta$-admissible frame $\{\tau_{\mu}\}_{\mu=1}^N$ is $\sigma(y,\eta,\lambda)$ given by \eqref{localsigmasym}, and it satisfies the symbol estimates \eqref{symbest}.

Making use of locality we can disregard the dependence of the coefficients of $A + \lambda^2$ in \eqref{AlExt} outside of the set $K\times\Z\times[0,c_0]$ and consider without loss of generality
\begin{gather*}
A + \lambda^2 : C_c^{\infty}(U;C^{\infty}(\open\Z^\wedge;E^\wedge)) \to C_c^{\infty}(U;C^{\infty}(\open\Z^\wedge;E^\wedge)), \\
A + \lambda^2 = x^{-2}\sum_{k+|\beta| \leq 2} a_{k,\beta}(y,x) (xD_y)^{\beta}(xD_x)^k + \lambda^2,
\end{gather*}
where $a_{k,\beta}(y,x) \in C^{\infty}(U\times\overline{\R}_+,\Diff^{2-|\beta|-k}(\Z;E_{\Z}))$ with $a_{k,\beta}(y,x) = 0$ for $(y,x) \notin K\times[0,c_0]$. We have
$$
A + \lambda^2 = \sum_{|\beta| \leq 2} a_{\beta}(y)D_y^{\beta} + \lambda^2,
$$
where
\begin{equation}\label{abeta}
a_{\beta}(y) = x^{-2+|\beta|}\sum_{k=0}^{2-|\beta|}a_{k,\beta}(y,x)(xD_x)^k \in x^{-2+|\beta|}\Diff_b^{2-|\beta|}(\Z^\wedge;E^\wedge).
\end{equation}
Let
\begin{equation}\label{absymbol}
a(y,\eta,\lambda) = \sum_{|\beta| \leq 2} a_{\beta}(y)\eta^{\beta} + \lambda^2 : C^{\infty}(\open\Z^\wedge;E^\wedge) \to C^{\infty}(\open\Z^\wedge;E^\wedge)
\end{equation}
be the boundary symbol associated with $A + \lambda^2$. The composition
$$
(A + \lambda^2)\circ \sym(y,D_y,\lambda) : C_c^{\infty}(U;\trb_F) \to C^{\infty}(U;C^{\infty}(\open\Z^\wedge;E^\wedge))
$$
is a pseudodifferential operator whose local symbol with respect to the chosen frame $\{\tau_{\mu}\}_{\mu=1}^N$ of $\trb_F$ is
\begin{equation}\label{asigmaexpand}
\begin{gathered}
(a \# \sigma)(y,\eta,\lambda) = c_0(y,\eta,\lambda) + c_1(y,\eta,\lambda) + c_2(y,\eta,\lambda) : \C^N \to C^{\infty}(\open\Z^\wedge;E^\wedge) \\
c_j = \sum\limits_{|\alpha| = j} \frac{1}{\alpha!} \bigl(\partial^{\alpha}_{\eta}a\bigr)\bigl(D_y^{\alpha}\sigma\bigr), \quad j=0,1,2.
\end{gathered}
\end{equation}
We will proceed to show below that each $c_j(y,\eta,\lambda)$ satisfies for all $s,t \in \R$ and all $\alpha \in {\mathbb N}_0^q$ and $\beta \in {\mathbb N}_0^{q+1}$ the estimate
\begin{equation}\label{cjestimate}
\|\kappa_{\langle \eta,\lambda \rangle}^{-1} \bigl(D_y^{\alpha}\partial^{\beta}_{(\eta,\lambda)}c_j(y,\eta,\lambda)\bigr) \langle \eta,\lambda \rangle^{g(y)}\|_{\L(\C^N,\K^{s,-1}_t)} \lesssim \langle \eta,\lambda \rangle^{2-j(1-\delta)-|\beta|+\delta|\alpha|}
\end{equation}
for all $(\eta,\lambda) \in \R^{q+1}$, locally uniformly in $y \in U$. Consequently,
$$
\|\kappa_{\langle \eta,\lambda \rangle}^{-1} \bigl(D_y^{\alpha}\partial^{\beta}_{(\eta,\lambda)}(a\#\sigma)(y,\eta,\lambda)\bigr) \langle \eta,\lambda \rangle^{g(y)}\|_{\L(\C^N,\K^{s,-1}_t)} \lesssim \langle \eta,\lambda \rangle^{2-|\beta|+\delta|\alpha|},
$$
which implies that
$$
(a\#\sigma)(y,D_y,\lambda) : H^{s'-g}_{\textup{comp}}(U;\C^{N}) \to {\mathcal W}^{s'-2}_{\textup{loc}}(U;\K^{s,-1}_t)
$$
for all $s',s,t \in \R$. Specifically for $s'=2$, $s$ arbitrary, and $t=s-(-1)-\frac{\dim\Z^\wedge}{2}$ the space ${\mathcal W}^{s'-2}_{\textup{loc}}(U;\K^{s,-1}_t)$ is a local model space for $x^{-1}H^s_e(\M;E)$ near the boundary, and consequently
$$
(A + \lambda^2)\circ {}_{\psi}(\Ext_{\lambda})_{\varphi} : H^{2-\gen}(\Y;\trb_F) \to x^{-1}H^{\infty}_e(\M;E)
$$
as desired.

It remains to prove the estimates \eqref{cjestimate}. To this end, note that each $a_{\beta}(y)$ in \eqref{abeta} belongs to
$$
C^{\infty}(U,\L(\K^{s,\nu}_t,\K^{s-(2-|\beta|),\nu-(2-|\beta|)}_{t}))
$$
for all $s,\nu,t \in \R$, and for all $\alpha \in {\mathbb N}_0^q$ we have
\begin{equation}\label{EstPartsA}
\|\kappa_{\langle \eta,\lambda \rangle}^{-1}\bigl(D^{\alpha}_ya_{\beta}(y)\bigr)\kappa_{\langle \eta,\lambda \rangle}\|_{\L(\K^{s,\nu}_t,\K^{s-(2-|\beta|),\nu-(2-|\beta|)}_{t})} \lesssim \langle \eta,\lambda \rangle^{2-|\beta|}
\end{equation}
for all $(\eta,\lambda) \in \R^{q+1}$, locally uniformly in $y \in U$.  These estimates show that for all $\alpha \in {\mathbb N}_0^q$ and all $\beta \in {\mathbb N}_0^{q+1}$ we have
$$
\|\kappa_{\langle \eta,\lambda \rangle}^{-1}\bigl(D^{\alpha}_y\partial^{\beta}_{(\eta,\lambda)}a(y,\eta,\lambda)\bigr)\kappa_{\langle \eta,\lambda \rangle}\|_{\L(\K^{s,0}_t,\K^{s-2+|\beta|,-2+|\beta|}_{t})} \lesssim \langle \eta,\lambda \rangle^{2-|\beta|}.
$$
Combining the latter estimate with the symbol estimates \eqref{symbest} for $\sigma(y,\eta,\lambda)$ then yields
$$
\|\kappa_{\langle \eta,\lambda \rangle}^{-1} \bigl(D_y^{\alpha}\partial^{\beta}_{(\eta,\lambda)}c_j(y,\eta,\lambda)\bigr) \langle \eta,\lambda \rangle^{g(y)}\|_{\L(\C^N,\K^{s,-1}_t)} \lesssim \langle \eta,\lambda \rangle^{2-j(1-\delta)-|\beta|+\delta|\alpha|}
$$
for $j = 1,2$. Observe that the target space can be chosen to be $\K^{s,-1}_t$ for $j=1,2$ because the symbol $a(y,\eta,\lambda)$ is differentiated at least once with respect to $\eta \in \R^q$ in each occurrence in the formula \eqref{asigmaexpand} for $c_1$ and $c_2$.

To prove \eqref{cjestimate} for $c_0(y,\eta,\lambda) = a(y,\eta,\lambda)\sigma(y,\eta,\lambda)$ we expand $a(y,\eta,\lambda)$ as in \eqref{absymbol} and write
\begin{equation}\label{c01}
\begin{gathered}
c_0(y,\eta,\lambda) = a_0(y)\sigma(y,\eta,\lambda) + a_1(y,\eta,\lambda)\sigma(y,\eta,\lambda), \\
a_1(y,\eta,\lambda) = \sum_{1 \leq |\beta| \leq 2} a_{\beta}(y)\eta^{\beta} + \lambda^2.
\end{gathered}
\end{equation}
The estimates \eqref{EstPartsA} imply that $a_1(y,\eta,\lambda)$ satisfies the symbol estimates
$$
\|\kappa_{\langle \eta,\lambda \rangle}^{-1}\bigl(D^{\alpha}_y\partial^{\beta}_{(\eta,\lambda)}a_1(y,\eta,\lambda)\bigr)\kappa_{\langle \eta,\lambda \rangle}\|_{\L(\K^{s,0}_t,\K^{s-2,-1}_{t})} \lesssim \langle \eta,\lambda \rangle^{2-|\beta|},
$$
and consequently
\begin{equation}\label{c02}
\|\kappa_{\langle \eta,\lambda \rangle}^{-1} \bigl((D_y^{\alpha}\partial^{\beta}_{(\eta,\lambda)}a_1\sigma)(y,\eta,\lambda)\bigr) \langle \eta,\lambda \rangle^{g(y)}\|_{\L(\C^N,\K^{s,-1}_t)} \lesssim \langle \eta,\lambda \rangle^{2-|\beta|+\delta|\alpha|}.
\end{equation}
Next consider $a_0(y) \in x^{-2}\Diff_b^2(\Z^\wedge;E^\wedge)$ as given by \eqref{abeta}. Taylor expansion of the coefficients $a_{k,0}(y,x)$ with respect to $x$ at $x=0$ gives
\begin{equation}\label{c03}
a_0(y) = \bA_y + \tilde{a}_0(y),
\end{equation}
where $\tilde{a}_0(y) \in x^{-1}\Diff_b^2(\Z^\wedge;E^\wedge)$, and $\bA_y = A_\wedge(0)$ is the indicial operator, i.e., the normal operator $A_\wedge(0)$ at $0 \in T^*_y\Y$. The operator $\tilde{a}_0(y)$ satisfies the estimates
$$
\|\kappa_{\langle \eta,\lambda \rangle}^{-1}\bigl(D^{\alpha}_y\tilde{a}_{0}(y)\bigr)\kappa_{\langle \eta,\lambda \rangle}\|_{\L(\K^{s,0}_{t+1},\K^{s-2,-1}_{t})} \lesssim \langle \eta,\lambda \rangle^{1}.
$$
Combined with \eqref{symbest} we get
\begin{equation}\label{c04}
\|\kappa_{\langle \eta,\lambda \rangle}^{-1} \bigl((D_y^{\alpha}\partial^{\beta}_{(\eta,\lambda)}\tilde{a}_0\sigma)(y,\eta,\lambda)\bigr) \langle \eta,\lambda \rangle^{g(y)}\|_{\L(\C^N,\K^{s,-1}_t)} \lesssim \langle \eta,\lambda \rangle^{1-|\beta|+\delta|\alpha|}.
\end{equation}
It remains to analyze $\bA_y\sigma(y,\eta,\lambda)$. By \eqref{symExtdef} we have
$$
\sym(y,\eta,\lambda)[\tau(y)] = [(z,x) \mapsto \psi(y)\omega(x[\eta,\lambda]_y)\jmath_y\tau(y,z,x)]
$$
for $\tau \in C^{\infty}(U;\trb_F)$. By definition of the bundle $\trb_F$ we have $\bA_y\jmath_y\tau \equiv 0$ for all $\tau \in C^{\infty}(U;\trb_F)$. Consequently,
\begin{align*}
\bA_y \sym(y,\eta,\lambda)[\tau(y)] &= [(z,x) \mapsto \psi(y)[\bA_y,\omega(x[\eta,\lambda]_y)]\tilde{\omega}(x[\eta,\lambda]_y)\jmath_y\tau(y,z,x)] \\
&= [\bA_y,\omega(x[\eta,\lambda]_y)] \Bigl([(z,x) \mapsto \psi(y)\tilde{\omega}(x[\eta,\lambda]_y)\jmath_y\tau(y,z,x)]\Bigr),
\end{align*}
where $\tilde{\omega} \in C_c^{\infty}(\overline{\R}_+)$ with $\tilde{\omega} \equiv 1$ in a neighborhood of the support of $\omega$. The commutator
$$
q(y,\eta,\lambda) = [\bA_y,\omega(x[\eta,\lambda]_y)] \in C^{\infty}(U\times\R^{q+1},\L(\K^{s,\nu}_t,\K^{s-2,\nu'}_{t'}))
$$
for all $s,\nu,\nu',t,t' \in \R$, and the symbol estimates
$$
\|\kappa_{\langle \eta,\lambda \rangle}^{-1}\bigl(D^{\alpha}_y\partial^{\beta}_{(\eta,\lambda)}q(y,\eta,\lambda)\bigr)\kappa_{\langle \eta,\lambda \rangle}\|_{\L(\K^{s,\nu}_t,\K^{s-2,\nu'}_{t'})} \lesssim \langle \eta,\lambda \rangle^{2-|\beta|}
$$
hold. We have $\bA_y\sigma(y,\eta,\lambda) = q(y,\eta,\lambda)\tilde{\sigma}(y,\eta,\lambda)$, where $\tilde{\sigma}(y,\eta,\lambda)$ is defined based on the cut-off function $\tilde{\omega}$ just like $\sigma(y,\eta,\lambda)$ is defined based on the cut-off function $\omega$. In particular, the estimates \eqref{symbest} also hold for $\tilde{\sigma}(y,\eta,\lambda)$ en lieu of $\sigma(y,\eta,\lambda)$. Combining these with the estimates on $q(y,\eta,\lambda)$ gives
\begin{equation}\label{c05}
\|\kappa_{\langle \eta,\lambda \rangle}^{-1} \big(D_y^{\alpha}\partial^{\beta}_{(\eta,\lambda)}\bA_y\sigma(y,\eta,\lambda)\big) \langle \eta,\lambda \rangle^{g(y)}\|_{\L(\C^N,\K^{s,-1}_t)} \lesssim \langle \eta,\lambda \rangle^{2-|\beta|+\delta|\alpha|}.
\end{equation}
In conclusion, starting with the representation \eqref{c01} for $c_0(y,\eta,\lambda)$, further broken up according to \eqref{c03}, we have obtained symbol estimates \eqref{c02}, \eqref{c04}, and \eqref{c05} which together give the desired estimates \eqref{cjestimate} for $c_0(y,\eta,\lambda)$. This completes the proof of the lemma.
\end{proof}

Let $S^{m}_{1,\delta}(U\times\R^{q+1},(\C^N,-g),(\K^{s,-1}_t,\kappa_{\varrho}))$, or merely $S^{m}_{1,\delta}$ for short, denote the symbol space of all smooth $\L(\C^N,\K^{s,-1}_t)$-valued functions $b(y,\eta,\lambda)$ on $U\times\R^{q+1}$ that satisfy for all $\alpha \in {\mathbb N}_0^q$ and all $\beta \in {\mathbb N}_0^{q+1}$ the estimate
\begin{equation}\label{bdrysymbest}
\|\kappa_{\langle \eta,\lambda \rangle}^{-1} \bigl(D_y^{\alpha}\partial^{\beta}_{(\eta,\lambda)}b(y,\eta,\lambda)\bigr) \langle \eta,\lambda \rangle^{g(y)}\|_{\L(\C^N,\K^{s,-1}_t)} \lesssim \langle \eta,\lambda \rangle^{m-|\beta|+\delta|\alpha|}
\end{equation}
for all $(\eta,\lambda) \in \R^{q+1}$, locally uniformly in $y \in U$. Using the notation of the previous proofs we have shown in Lemma~\ref{ExtAdvMapProps} that
$$
(a\#\sigma)(y,\eta,\lambda) \in S^{2}_{1,\delta}(U\times\R^{q+1},(\C^N,-g),(\K^{s,-1}_t,\kappa_{\varrho})).
$$
More precisely, we have $(a\#\sigma)(y,\eta,\lambda) = c_0(y,\eta,\lambda)$ mod $S^{1+\delta}_{1,\delta}$ by \eqref{asigmaexpand} and the estimates \eqref{cjestimate}. We further analyzed $c_0(y,\eta,\lambda)$ by breaking it up as in \eqref{c01} as
$$
c_0(y,\eta,\lambda) = a_0(y)\sigma(y,\eta,\lambda) + a_1(y,\eta,\lambda)\sigma(y,\eta,\lambda),
$$
and have shown that $a_0(y)\sigma(y,\eta,\lambda) = \bA_y\sigma(y,\eta,\lambda)$ modulo $S^{1+\delta}_{1,\delta}$, see the estimates \eqref{c04} and \eqref{c05}. We accomplished the latter by using a Taylor expansion with respect to $x$ at $x=0$ of the coefficients $a_{k,0}(y,x)$ in $a_0(y)$ given by \eqref{abeta}. Using a Taylor expansion on the coefficients in \eqref{abeta} also for $\beta \neq 0$ and arguing as as we did in the proof of Lemma~\ref{ExtAdvMapProps} for $a_0(y)\sigma(y,\eta,\lambda)$, but for $a_1(y,\eta,\lambda)\sigma(y,\eta,\lambda)$ instead, then reveals that
$$
c_0(y,\eta,\lambda) = \bigl(A_\wedge(\eta) + \lambda^2\bigr)\sigma(y,\eta,\lambda)
$$
modulo $S^{1+\delta}_{1,\delta}$, and thus
$$
(a\#\sigma)(y,\eta,\lambda) = \bigl(A_\wedge(\eta) + \lambda^2\bigr)\sigma(y,\eta,\lambda)
$$
modulo $S^{1+\delta}_{1,\delta}$. The twisted homogeneity relations \eqref{kappahomogeneity} for $A_\wedge(\eta)$ and \eqref{kappahomsym} for $\sym(y,\eta,\lambda)$ show that
\begin{multline*}
\bigl(A_\wedge(\varrho\eta) + (\varrho\lambda)^2\bigr)\sym(y,\varrho\eta,\varrho\lambda) =
\varrho^2\kappa_{\varrho}\bigl(A_\wedge(\eta) + \lambda^2\bigr)\sym(y,\eta,\lambda)\kappa_{\varrho}^{-1} \\
: \trb_{F,y} \to C^{\infty}(\open\Z^\wedge;E^\wedge)
\end{multline*}
for all sufficiently large $|\eta|_y^2+\lambda^2$ and all $\varrho \geq 1$. Setting
\begin{equation*}
\sym_\wedge(y,\eta,\lambda)[\tau] = \Bigl[(z,x) \mapsto \psi(y)\omega\Bigl(x\sqrt{|\eta|_y^2 +\lambda^2}\Bigr)\jmath_y\, \tau(y,z,x)\Bigr]
\end{equation*}
for $y \in U$, $(0,0) \neq (\eta,\lambda) \in \R^{q+1}$, and $\tau \in C^{\infty}(U;\trb_F)$ in analogy with the  definition of the full symbol $\sym(y,\eta,\lambda)$ in \eqref{symExtdef}, we obtain the twisted homogeneous principal symbol $\sym_\wedge(y,\eta,\lambda)$ associated with $\sym(y,\eta,\lambda)$. This principal symbol satisfies
$$
\sym_\wedge(y,\varrho\eta,\varrho\lambda) = \kappa_{\varrho}\sym_\wedge(y,\eta,\lambda)\kappa_{\varrho}^{-1}
$$
for all $y \in U$, $(\eta,\lambda) \in \R^{q+1}\minus \{(0,0)\}$. Thus also $(a\#\sigma)(y,\eta,\lambda)$ has a twisted homogeneous principal symbol $(a\#\sigma)_\wedge(y,\eta,\lambda)$ given by
$$
(a\#\sigma)_\wedge(y,\eta,\lambda)\begin{bmatrix} c^1 \\ \vdots \\ c^N \end{bmatrix} =
\bigl(A_\wedge(\eta) + \lambda^2\bigr)\sym_\wedge(y,\eta,\lambda)\Bigl[\sum\limits_{\mu=1}^N c^{\mu}\tau_{\mu}(y)\Bigr]
$$
for $\begin{bmatrix} c^1 & \cdots & c^{N} \end{bmatrix}^\dag \in \C^N$, and the principal symbol $(a\#\sigma)_\wedge(y,\eta,\lambda)$ determines $(a\#\sigma)(y,\eta,\lambda)$ modulo $S^{1+\delta}_{1,\delta}$.

\begin{definition}
The normal family of the parameter-dependent family of extension operators $\Ext_{\lambda} : C^{-\infty}(\Y;\trb_F) \to H^{\infty}_e(\M;E)$ is the family
\begin{equation*}
\Ext_\wedge(\pmb\eta,\lambda)=\omega(x\sqrt{g_\Y(\pmb\eta)+\lambda^2})\,\jmath_y : \trb_{F,y} \to C^{\infty}(\Z^\wedge_y;E_{\Z^\wedge_y}),\quad(\pmb\eta,\lambda) \in \bigl(T_y^*\Y\times\R\bigr)\minus 0.
\end{equation*}
\end{definition}

By the previous lemmas and remarks, the composition
$$
(A+\lambda^2)\circ\Ext_{\lambda} : H^{2-\gen}(\Y;\trb_F) \to x^{-1}H^{\infty}_e(\M;E)
$$
is, after appropriate trivialization of the bundles locally near the boundary, given by pseudodifferential operators with parameter-dependent operator valued symbols of class $S^{2}_{1,\delta}(U\times\R^{q+1},(\C^N,-g),(\K^{s,-1}_t,\kappa_{\varrho}))$, and the normal family
$$
(A_\wedge(\pmb\eta) + \lambda^2)\circ \Ext_\wedge(\pmb\eta,\lambda) : \trb_{F,y} \to C^{\infty}(\Z^\wedge_y;E_{\Z^\wedge_y})
$$
determines these boundary symbols modulo $S^{1+\delta}_{1,\delta}(U\times\R^{q+1},(\C^N,-g),(\K^{s,-1}_t,\kappa_{\varrho}))$.

\begin{proposition}\label{Parametrix}
Under the assumptions (\ref{SymmetricSemibounded})-(\ref{AhatInvertible}) in Section~\ref{sec-SetupAssumptions} on the operator $A \in x^{-2}\Diff^2_e(\M;E)$, the operator family
\begin{equation}\label{paramopfam}
\begin{bmatrix} A + \lambda^2 & (A+\lambda^2)\circ\Ext_{\lambda} \end{bmatrix} :
\begin{array}{c} x^1H^2_e(\M;E) \\ \oplus \\ H^{2-\gen}(\Y;\trb_F) \end{array} \to x^{-1}L^2_b(\M;E)
\end{equation}
is invertible for all sufficiently large $|\lambda| \gg 0$.
\end{proposition}
\begin{proof}
By the conditions (\ref{SymmetricSemibounded})-(\ref{AhatInvertible}), the parameter-dependent homogenous $w$-principal symbol
$\wsym(A)(\pmb\xi) + \lambda^2$ of the operator family $A + \lambda^2$ is invertible on $\bigl(\wT^*\M\times\R\bigr)\minus 0$, see \eqref{wsymparamell}. Moreover, the parameter-dependent normal family
$$
\begin{bmatrix} A_\wedge(\pmb\eta) + \lambda^2 & (A_\wedge(\pmb\eta) + \lambda^2)\circ\Ext_\wedge(\pmb\eta,\lambda) \end{bmatrix} : \begin{array}{c} \Dom_{\wedge,\min} \\ \oplus \\ \trb_{F,y} \end{array} \to x^{-1}L^2_b(\Z^\wedge_y,E_{\Z_y^\wedge})
$$
is invertible on $\bigl(T^*\Y\times\R\bigr)\minus 0$. The latter follows because the normal family of Friedrichs extensions
$$
A_\wedge(\pmb\eta) + \lambda^2 : \Dom_{\wedge,F} \to x^{-1}L^2_b(\Z^\wedge_y,E_{\Z_y^\wedge})
$$
is invertible on $\bigl(T^*\Y\times\R\bigr)\minus 0$ by \eqref{normalfamparamell}, and because we have
$$
\begin{bmatrix} A_\wedge(\pmb\eta) + \lambda^2 & (A_\wedge(\pmb\eta) + \lambda^2)\circ\Ext_\wedge(\pmb\eta,\lambda) \end{bmatrix} = \bigl(A_\wedge(\pmb\eta) + \lambda^2\bigr)\circ\begin{bmatrix} \iota & \Ext_\wedge(\pmb\eta,\lambda) \end{bmatrix},
$$
where
$$
\begin{bmatrix} \iota & \Ext_\wedge(\pmb\eta,\lambda) \end{bmatrix} : \begin{array}{c} \Dom_{\wedge,\min} \\ \oplus \\ \trb_{F,y} \end{array} \to \Dom_{\wedge,F}
$$
is bijective on $\bigl(T^*\Y\times\R\bigr)\minus 0$. The latter isomorphism, as was discussed in Section~\ref{sec-SetupAssumptions}, is a consequence of the main result of \cite{GiMe01} characterizing the domain of the Friedrichs extension for elliptic cone operators. These observations, combined with the assumption on the indicial roots of $A$ in (\ref{b-strips}) of Section~\ref{sec-SetupAssumptions} and the careful structural analysis of the operator family $(A+\lambda^2)\circ\Ext_{\lambda}$ carried out in the proofs of Lemmas~\ref{ExtBasicMapProps} and \ref{ExtAdvMapProps}, enable us to construct a parameter-dependent parametrix
\begin{equation}\label{paramopparam}
\begin{bmatrix}
P(\lambda) \\ T(\lambda)
\end{bmatrix} : x^{-1}L^2_b(\M;E) \to \begin{array}{c} x^1H^2_e(\M;E) \\ \oplus \\ H^{2-\gen}(\Y;\trb_F) \end{array} \end{equation}
to the operator family \eqref{paramopfam}, utilizing (local) symbolic inversion, quantization, patching, and a formal Neumann series argument along the lines of the methods developed in Schulze's parameter-dependent edge calculus (see \cite[Chapters 2 and 3]{Dorschfeldt}, \cite[Chapter 3]{SchuNH}, or \cite[Chapter 3]{SchuWiley}; the reader may also consult \cite[Section 6]{GiKrMe10}). 

There are some minor differences of a merely technical nature in our situation when compared to those references pertaining to the boundary symbolic and operator calculus. Namely, in said references only standard Sobolev spaces of sections appear on $\Y$. If the reader desires to reduce to that situation, the operator family \eqref{paramopfam} must be composed from the right (over $\Y$) with an invertible family of parameter-dependent operators $R(\lambda)$ belonging to the parameter-dependent version of the calculus of \cite{KrMe12b}. This yields an operator family
\begin{equation}\label{opfamparam1}
\begin{bmatrix} A + \lambda^2 & K(\lambda) \end{bmatrix} :
\begin{array}{c} x^1H^2_e(\M;E) \\ \oplus \\ H^{2}(\Y;\trb_F) \end{array} \to x^{-1}L^2_b(\M;E),
\end{equation}
where $K(\lambda) = (A+\lambda^2)\circ\Ext_{\lambda}\circ R(\lambda)$. The operator family $K(\lambda)$, and along with it the full operator matrix \eqref{opfamparam1}, then almost belongs to the class discussed in the referenced literature, but is locally near the boundary based on operator-valued symbols of class $1,\delta$ for some $\delta > 0$, whereas in Schulze's calculus only type $1,0$ symbols are considered. That this modification does not cause any changes to the results follows from \cite{Kr15}.

Alternatively, the reader may from the very beginning consider a generalized parameter-dependent boundary symbolic and operator calculus by allowing group actions generated by bundle endomorphisms on the bundles over $\Y$, and incorporate those actions into symbol estimates of type $1,\delta$ in the form stated in \eqref{bdrysymbest}. All results of Schulze's calculus generalize to this wider class. That this is indeed the case only requires the pseudodifferential calculi introduced in \cite{KrMe12b,Kr15}. In particular, the wider class of parameter-dependent edge pseudodifferential operators then contains the operator family \eqref{paramopfam} and its parametrix \eqref{paramopparam} directly.

The parametrix \eqref{paramopparam} now yields an inverse of \eqref{paramopfam} modulo operator families that are rapidly decreasing in the parameter $\lambda \in \R$, i.e.,
\begin{align*}
\begin{bmatrix} A + \lambda^2 & (A+\lambda^2)\circ\Ext_{\lambda} \end{bmatrix} \circ
\begin{bmatrix} P(\lambda) \\ T(\lambda) \end{bmatrix} - \Id &\in {\mathscr S}(\R,\L(x^{-1}L^2_b(\M;E))), \\
\begin{bmatrix} P(\lambda) \\ T(\lambda) \end{bmatrix} \circ
\begin{bmatrix} A + \lambda^2 & (A+\lambda^2)\circ\Ext_{\lambda} \end{bmatrix}
- \begin{bmatrix} \Id & 0 \\ 0 & \Id \end{bmatrix} &\in {\mathscr S}\left(\R,\L\left(\begin{array}{c} x^1H^2_e(\M;E) \\ \oplus \\ H^{2-\gen}(\Y;\trb_F) \end{array}\right)\right).
\end{align*}
Consequently, the invertibility of \eqref{paramopfam} for large $|\lambda|$ follows, and the proposition is proved.
\end{proof}


\end{document}